\documentclass[11pt, oneside]{amsart}
\usepackage{amsfonts, amstext, amsmath, amsthm, amscd, amssymb}
\usepackage{manfnt}

\usepackage{url}
\usepackage[all]{xypic}

\usepackage{tikz}
\usepackage{graphicx}
\usepackage{multido}

\usepackage{graphics, pinlabel, color}
\usepackage{comment}
\usepackage{enumerate}

\usepackage[all,graph]{xy}
\usepackage{float, placeins,  longtable}

\renewcommand{\setminus}{{\smallsetminus}}

\newcommand{\bp}{\begin{pmatrix}}
\newcommand{\ep}{\end{pmatrix}}
\newcommand{\be}{\begin{equation}}
\newcommand{\ee}{\end{equation}}
\newcommand{\ol}[1]{\overline{#1}}

\numberwithin{equation}{section}

\theoremstyle{plain}
\newtheorem{theorem}[equation]{Theorem}
\newtheorem{lemma}[equation]{Lemma}
\newtheorem{proposition}[equation]{Proposition}

\newtheorem{corollary}[equation]{Corollary}

\newtheorem{prop}[equation]{Proposition}
\newtheorem*{claim*}{Claim}

\theoremstyle{definition}
\newtheorem{example}[equation]{Example}
\newtheorem{remark}[equation]{Remark}

\newtheorem{definition}[equation]{Definition}

\newtheorem{scholium}[equation]{Scholium}

\numberwithin{equation}{section}

 \newtheoremstyle{TheoremNum}
        {}{}              
        {\itshape}                      
        {}                              
        {\bfseries}                     
        {.}                             
        { }                             
        {\thmname{#1}\thmnote{ \bfseries #3}}
\theoremstyle{TheoremNum}
\newtheorem{propn}{Proposition}

\newtheorem{thmn}{Theorem}

\def\N{\mathbb{N}}
\def\Z{\mathbb Z}
\def\R{\mathbb R}
\def\Q{\mathbb Q}
\def\C{\mathcal{C}}

\def\wt#1{\widetilde{#1}}

\def\sm{\setminus}
\def\a{\alpha}

\def\toiso{\xrightarrow{\simeq}}

\def\bp{\begin{pmatrix}}
\def\ep{\end{pmatrix}}
\def\ba{\begin{array}}
\def\ea{\end{array}}
\def\bn{\begin{enumerate}}
\def\en{\end{enumerate}}

\def\NG{\mathcal{N}\Gamma}
\DeclareMathOperator\Int{Int}

\DeclareMathOperator\Arf{Arf}

\DeclareMathOperator\cl{cl}
\DeclareMathOperator\Wh{Wh}

\DeclareMathOperator\Id{Id}

\DeclareMathOperator\im{im}
\DeclareMathOperator\gw{gw}

\newcommand{\smfrac}[2]{\mbox{\footnotesize$\displaystyle\frac{#1}{#2}$}} 

\begin{document}

\title[]{Grope metrics on the knot concordance set}

\author{Tim D. Cochran$^{\S}$}
\address{Rice University, Houston, TX, USA}

\author{Shelly Harvey$^{\dag}$}
\address{Rice University, Houston, TX, USA}
\email{shelly@rice.edu}

\author{Mark Powell$^\ddag$}
\address{
Universit\'e du Qu\'ebec \`a Montr\'eal, QC, Canada
}
\email{mark@cirget.ca}

\thanks{$^{\S}$ The first author was partially supported by National Science Foundation grant DMS-1309081 and by a grant from the Simons Foundation (\#304603 to Tim Cochran).  $^{\dag}$ The second author was partially supported by National Science Foundation grants DMS-1309070 and DMS-1309081 and by a grant from the Simons Foundation (\#304538 to Shelly Harvey). $^\ddag$ The third author was partially supported by an NSERC grant entitled ``Structure of knot and link concordance.''}

\def\subjclassname{\textup{2010} Mathematics Subject Classification}
\expandafter\let\csname subjclassname@1991\endcsname=\subjclassname
\expandafter\let\csname subjclassname@2000\endcsname=\subjclassname

\begin{abstract}
To a special type of grope embedded in $4$-space, that we call a branch-symmetric grope, we associate a length function for each real number $q \geq 1$.  This gives rise to a family of pseudo-metrics $d^q$, refining the slice genus metric, on the set of concordance classes of knots, as the infimum of the length function taken over all possible grope concordances between two knots.  We investigate the properties of these metrics.  The main theorem is that the topology induced by this metric on the knot concordance set is not discrete for all $q>1$.  The analogous statement for links also holds for $q=1$. In addition we translate much previous work on knot concordance into distance statements. In particular, we show that winding number zero satellite operators are contractions in many cases, and we give lower bounds on our metrics arising from knot signatures and higher order signatures.  This gives further evidence in favor of the conjecture that the knot concordance group has a fractal structure.
\end{abstract}
\maketitle

\section{Introduction}

Let $\C$ be the set of smooth concordance classes of knots.
This is in fact a group under connected sum, but in this paper we will primarily consider $\C$ as a real valued metric space, and study satellite-type operators on $\C$ which are typically not homomorphisms.

An explicit study of metrics on the set of concordance classes of knots was begun by the first and second authors~\cite{Cochran-Harvey:2014-1}.  Metrics were studied arising from the slice genus, and the homology norm, which measures the second Betti number of a 4-manifold in which two knots $K$ and $J$ cobound an annulus.  One drawback of this study is that distances between knots are all integer valued.

Below we define, for each real number $q \geq 1$, a new pseudo-metric, $d^q$, which refines the slice genus by taking advantage of deeper, higher order geometric information in order to define rational (and potentially real) valued notions of distance.  A length is associated to a special type of grope that is built out of symmetric gropes, by counting the genera of the component surfaces.  The distance between two knots is then defined as the infimum of the lengths of all possible grope concordances between the two knots.   We will often omit the adjective ``pseudo.''

We show that our metrics capture much subtlety of knot concordance, in the following sense.  Let $U$ be the unknot, and let $\C^{top}$ be the set of topologically locally flat concordance classes of knots. For $q>1$, the distance between two topologically concordant knots vanishes, whence we obtain an induced pseudo-metric $d^q_{top}$ on $\C^{top}$.

\begin{thmn}[\ref{cor:notdiscrete}]
  For any $q>1$ there exist uncountably many sequences of knots $\{K_i\}_{i \geq 0}$ such that $d^q(K_i,U) > 0$ for all $i$ but $d^q(K_i,U) \to 0$ as $i \to \infty$.
In particular, if $q>1$ then neither of the topologies on $\C$ and $\C^{top}$, induced by $d^q$ and $d^q_{top}$ respectively, are discrete.
\end{thmn}

For links of at least two components, the analogous result holds for $q=1$ as well in the smooth case; see Section~\ref{section:links-q-equals-one}.

We stated the second sentence of the theorem for both the smooth and topological cases, but the proof for the smooth case is relatively straightforward.  A pseudo-metric that is not a metric certainly does not induce a discrete topology, and for $q>1$, every topologically slice knot has distance zero from the unknot (see Remark~\ref{remark:top-slice-knots}), so in fact $d^q$ only induces a pseudo-metric on $\C$, and not a metric.
Thus in the smooth case the second sentence of Theorem~\ref{cor:notdiscrete} can be deduced without the first sentence, whereas in the topological case one of the sequences of knots from the first sentence of the theorem seems to be required.  Even for $\C^{top}$, we are not able to show that $d^q$ is an honest metric.  However, in order to show non-discreteness in this way, one would have to find a non-topologically slice knot that lies in the intersection of the grope filtration of knot concordance i.e.\ that bounds a height $n$ symmetric grope for every~$n$.  Since this latter question remains open, one needs the sequences of knots of Theorem~\ref{cor:notdiscrete} to show that $\C^{top}$ is non-discrete when $q>1$.

Next, in addition to studying the metric space, we consider the action of natural operators on $(\mathcal{C}, d^q)$ called satellite operators.  We show that winding number zero operators are often contraction operators.   This gives further evidence that $\mathcal{C}$ has the structure of a fractal space as conjectured in \cite{CHL5}.

Let $\C^{\ell}$ be the set of concordance classes of $m$-component links with all pairwise linking numbers vanishing.  We will extend our metric to this set.  Let $\mathcal{C}^m_{SL}$ be the concordance group of $m$-component string links, also with pairwise linking numbers vanishing.

\begin{propn}[\ref{prop:contraction}]
For any winding number zero string link operator $R(-,\eta)$ there is an $N$, depending only on the geometric winding numbers of $R(-,\eta)$, such that for each $q>N$, $R\colon (\C^m_{SL},d^q) \to (\C^\ell,d^q)$ is a contraction mapping. In particular, for any winding number zero pattern knot $P$, and any $q$ greater than the geometric winding number of $P$, the satellite operator $P\colon (\C,d^q)\to(\C,d^q)$ is a contraction mapping.
\end{propn}

The same holds if $R$ is a string link, and the codomain of the operator $R(-,\eta)$ becomes $\C^\ell_{SL}$.

Of course there are many possible metrics of this character which one could define on the knot concordance set; perhaps the reader will come up with his or her own.  However, we think that any reasonable notion ought to satisfy certain meta-properties, which are satisfied by our metrics.  These are as follows.
\begin{itemize}
  \item Refines the slice genus metric.  In particular the slice genus is an upper bound.
  \item Winding number zero operators are contractions.
  \item Reflects the known complexity of knot concordance; has a relation to the algebraic concordance group, Casson-Gordon invariants, and the $n$-solvable filtration.
  \item Induces a non-discrete topology.
\end{itemize}

One may wonder why we do not consider metrics defined using half-gropes rather than symmetric gropes.
After all, in order to surger a surface to a disk, one only requires that a half basis of curves bounds framed disjointly embedded disks; the other half of first homology need not bound any surface at all.  However beware that all knots with Arf invariant zero bound half-gropes of arbitrarily large order~\cite{Schneiderman:2005}, therefore all Arf invariant zero knots would likely have very small norm in such a metric space.  Also, the construction of topological disks in Freedman-Quinn~\cite{FQ} uses symmetric gropes.   Dual spheres have an extra significance in dimension 4: the slogan of \cite[Chapter~5]{FQ} is that disks can be embedded topologically when there is a good fundamental group, in the presence of dual spheres.  So it seems to be natural to only consider symmetric gropes as being close approximations to disks.

Another refinement of the slice genus already in the literature is the stable slice genus of C.~Livingston~\cite{Li11}. However this is not known to satisfy the properties above other than the first.  In particular, no knot is known to have nonzero stable slice genus less than $1/2$.

In a future paper we will consider another metric on the smooth knot concordance group which has a closer relationship to the bipolar filtration~\cite{Cochran-Harvey-Horn:2013-1} and ``smooth invariants'' of knot concordance, which do not necessarily vanish on topologically slice knots, such as those invariants arising from Heegaard Floer theory, Khovanov homology, and contact topology.

\subsection*{Organization of the paper}

Section~\ref{section:defn-of-metric} gives the definition of our metrics and many of their important properties.
Section~\ref{section:relations-to-knot-signatures} proves that the knot signatures give lower bounds.
Section~\ref{section:defns-string-links} gives definitions involving string links and string link infections which are needed for the rest of the paper.
Section~\ref{sec:infections} investigates the effect of string link infections and satellite operators on the metrics.  It is shown that winding number zero operators are contractions whenever $q$ is larger than the geometric winding number of the operator.
Examples of knots which yield sequences of knots exhibiting the non-discrete behaviour of $\C$ are constructed in Section~\ref{section:examples-for-non-discrete}.
The analogous examples for the link case are given in Section~\ref{section:links-q-equals-one}.
Section~\ref{section:higher-order-rho-invariants} then proves lower bounds arising from higher order $L^{(2)}$ $\rho$-invariants.
Finally Section~\ref{section:quasi-isometries} shows that the identity map is not a quasi-isometry between our metric spaces and knot concordance with the slice genus metric.

\subsection*{Concerning Tim Cochran} We are saddened by the loss of the first author, Tim Cochran, who passed away on December 16, 2014, before the completion of this paper.   Tim played an extremely important role in this paper and wrote a large portion of it.   However, since he was neither able to verify nor influence the final version, any errors in the paper should be attributed to the other two authors.

\subsection*{Acknowledgements}

We would like to thank Maciej Borodzik, Jae Choon Cha, Stefan Friedl, Kent Orr, Rob Schneiderman, Peter Teichner, and Diego Vela for their valuable input during discussions on our grope metrics.
We thank the referee for a careful reading and many excellent, thoughtful comments that helped us to improve the exposition.

We are grateful to the Max Planck Institute for Mathematics in Bonn, as many of the results were obtained while all three of us were visitors there, and the paper was written while the third author was a visitor.

\section{Definition of the metric}\label{section:defn-of-metric}

A \emph{grope} is a special type of 2-complex, with a decomposition into a union of finitely many stages $\bigcup_k \, X_k$, and a specified boundary.
  Each stage $X_k$ is a union of surfaces with boundary, whose interiors are disjoint. When $k \geq 2$ each surface has exactly one boundary component.  The boundary of $X_1$ is the boundary of the grope.  The intersection $X_k \cap X_j =\emptyset$ for $|k-j|>1$, and the intersection $X_k \cap X_{k+1}$, for $k \geq 1$, is equal to the boundary of $X_{k+1}$ and forms a subset of a standard symplectic basis for the first homology of $X_k$.

 We will work with two types of gropes.  First we define \emph{symmetric gropes}, then we use these to define \emph{branch-symmetric} gropes.
A \emph{symmetric grope} has a \emph{height} $n \in \N$.
For $n = 1$ a symmetric grope is precisely a compact, oriented surface $G_1$ with a single boundary component on each connected component; the boundary components are called the base circles.  A symmetric grope $G_{n+1}$ of \emph{height} $n+1$ is defined
  inductively as follows: take a height one symmetric grope $G_1$ and let $\{\alpha_j\,|\, j = 1, \ldots, 2
   g(G_1)\}$ be a standard symplectic basis of circles
  for the first homology of~$G_1$ where $g(G_1) = \text{genus}(G_1)$. Then a symmetric grope of height $n+1$ is formed by attaching a
  connected symmetric grope $G^j_n$ of height $n$ to each $\alpha_j$ along its base circle; the base circles of $G_{n+1}$ are the boundary components of $G_1$.
The \emph{$k$th stage} of the symmetric grope is the union of the surfaces that were introduced by the $(n-k+1)$th inductive step in the construction, where taking a height one symmetric grope counts as the first step.  We denote the stages of a symmetric grope from $p$ through $q$ inclusive by~$G_{p: q}$.

Next we define a \emph{branch-symmetric grope}.  Let $\Sigma_{1: 1}$ be a compact oriented surface with either one or two boundary components on each connected component and let $\{\alpha_j\,|\, j = 1, \ldots, 2
 g(\Sigma_{1:1})\}$ be a standard symplectic basis of circles for the first homology of $\Sigma_{1: 1}$ where $\a_{2i-1}$and $\a_{2i}$ form a dual pair for each $1\leq i \leq g(\Sigma_{1:1})$. For each $\alpha_j$ for $j = 1, \ldots, 2
  g(\Sigma_{1:1})$, attach a connected symmetric grope $G^j_{m_j}$ of some height $m_j \geq 0$ to $\alpha_j$, no subsurface of which is a disk, and such that $m_{2i} = m_{2i-1}$ for $1\leq i \leq g(\Sigma_{1:1})$.  For this purpose we take a grope of height $0$ to be the empty set.     This defines a branch-symmetric grope $\Sigma$.  It may be that $\Sigma$ is not a symmetric grope.

In the above definition, we call the union of the two gropes attached to a single dual pair of basis curves $\alpha_{2i}$ and $\alpha_{2i-1}$ a \emph{branch} of $\Sigma$.  Let $m_j(\Sigma)$ be the height of the grope attached to $\alpha_j$.
We define the \emph{length} of the branch to be $n_i(\Sigma):= m_{2i}(\Sigma)= m_{2i-1}(\Sigma)$.
When it is clear, we write $n_j$ instead of $n_j(\Sigma)$.
Note that for each $i$, $\alpha_{2i-1}$ and $\alpha_{2i}$ must each have a grope of the same height attached to them, but the various surface stages may have different genera (they do not have to be the homeomorphic gropes).  We say that the \emph{$(k-1)$th stage surfaces} $G^j_{(k-1) : (k-1)}$ of the symmetric gropes $G^j_{n_j}$ are the $k$th stage surfaces $\Sigma_{k: k}$ of $\Sigma$.

For a branch-symmetric grope $\Sigma$, let $g_1(\Sigma)$ be the genus of $\Sigma_{1 : 1}$.  For each $i=1,\dots,g_1(\Sigma)$, we have a dual pair of basis curves $\a_{2i-1}$ and $\a_i$.  Define $g_2^i(\Sigma)$ to be the sum of the genera of the two second stage surfaces which are attached to these basis curves $\a_{2i-1}$ and $\a_{2i}$.
We denote the union of these two surfaces by $\Sigma_{2: 2}^i$.  If $n_i(\Sigma)=0$ then $g_2^i(\Sigma)=0$. Suppose $n_i(\Sigma)\geq 1$. Note that in this case $g_2^i(\Sigma)\geq 2$ since we do not allow disks. Inductively let $g_k^i(\Sigma)$, for $3\leq k\leq n_i(\Sigma)+1$, be the sum of the genera of the surfaces at the $k$th stage which are attached to the basis curves of $\Sigma_{(k-1):(k-1)}^i$.  Since $g_{k-1}^i(\Sigma)$ is the genus of  the collection $\Sigma_{(k-1):(k-1)}^i$, there are precisely $2g_{k-1}^i(\Sigma)$ surfaces at the $k$th stage which are attached to the basis curves of $\Sigma_{(k-1):(k-1)}^i$.  We denote the union of these surfaces by  $\Sigma_{k:k}^i$.  When it is clear, we will drop the $\Sigma$ and  write $g_k^i$ (respectively $g_1$) instead of $g_k^i(\Sigma)$ (respectively $g_1(\Sigma)$).

\begin{lemma}\label{lemma:counting-genera}
Let $\Sigma$ be a branch-symmetric grope.  For each  $1 \leq i \leq g_1(\Sigma)$ and $2\leq k\leq n_i(\Sigma)+1$ we have
$$
g_k^i(\Sigma) \geq 2g_{k-1}^i (\Sigma) \geq 2^{k-1}.
$$
\end{lemma}

\begin{proof}
Each surface has genus at least one.  The second inequality follows from induction.
\end{proof}

We say that two oriented knots $K$ and $J$ in $S^3$ \emph{cobound a branch-symmetric grope $\Sigma$} if $\Sigma$ is a branch-symmetric grope, $\Sigma_{1 : 1}$ has two boundary components, and we have a framed smooth embedding $\iota\colon \Sigma \hookrightarrow S^3 \times I$ where $\iota^{-1}(\partial (S^3 \times I)) = \partial \Sigma_{1:1}$, $\iota_{| {\iota^{-1}(S^3 \times \{0\})}} = K$ and $\iota_{| {\iota^{-1}(S^3 \times \{1\})}} = -J$.  We also say that $K$ and $J$ are \emph{grope concordant via $\Sigma$}.

Let $q \geq 1$ be a real number.  We define a metric for each $q\geq 1$ as follows.  For a branch-symmetric grope $\Sigma$, define
\begin{equation}\label{eq:knot_grope_formula} \|\Sigma\|^q :=   \sum_{i=1}^{g_1(\Sigma)}\frac{1}{q^{n_{i}(\Sigma)}} \bigg( 1- \sum_{k=2}^{n_{i}(\Sigma)+1} \frac{1}{g_k^{i}(\Sigma)}\bigg).\end{equation}
 Now we set
\[d^q(K,J) := \inf_{\Sigma}\{\|\Sigma\|^q \mid K\text{ and }J \text{ cobound a branch-symmetric grope } \Sigma \text{ in } S^3\}.\]

\begin{remark}~
\begin{enumerate}
\item
As a basic example, a surface $\Sigma$ of genus $g$, with no higher surfaces, has $\|\Sigma\|^q = g$ for any $q \geq 1$.
In particular the slice genus of $K$ is an upper bound for $d^q(K,U)$, where $U$ is the unknot.
  Any slice genus one knot with Arf invariant one has distance one from the unknot, since the Arf invariant obstructs the knot from bounding any height~2 grope by~\cite[Theorem~8.11~and~Remark~8.2]{COT}.

\item
Note that for constant $g_1(\Sigma)$, and $q>1$, the metric $d^q_{\Sigma}$ approaches zero for high height gropes.  For $q>1$, the same is true if every surface has genus one, since then the grope distance is \[1 - \smfrac{1}{2} - \smfrac{1}{4} - \smfrac{1}{8} - \cdots - \smfrac{1}{2^{n_i(\Sigma)}}.\]

Here is a concrete example a computation of $\|\Sigma\|^q$.
Consider the grope $\Sigma$ shown in Figure~\ref{figure:example-grope}

\begin{figure}[h]
\begin{center}
\begin{tikzpicture}
\node[anchor=south west,inner sep=0] at (0,0){\includegraphics[scale=0.7]{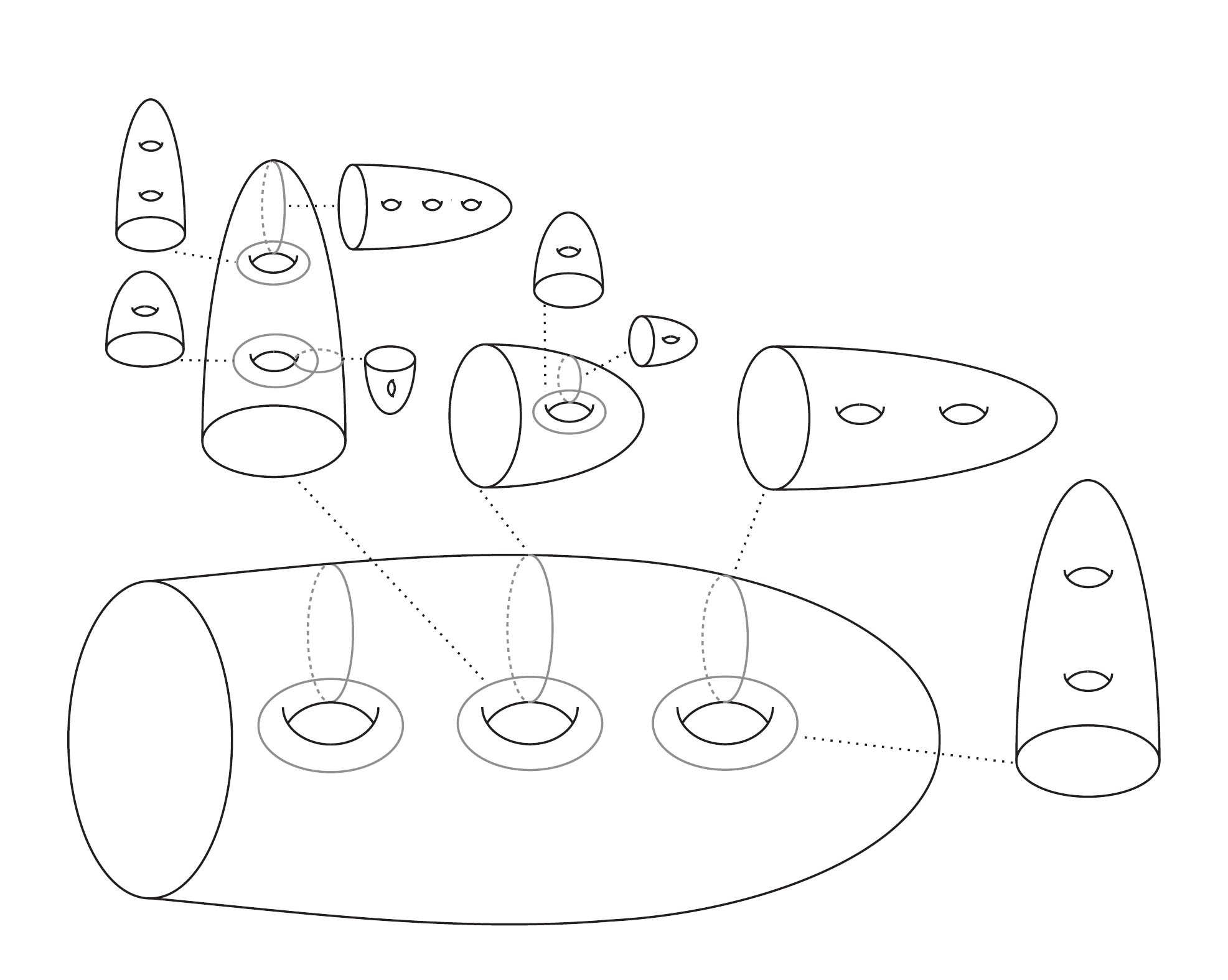}};
\node at (3.05,2.35)  {$\alpha_1$};
\node at (4.3,4)  {$\alpha_2$};
\node at (5.3,2.35)  {$\alpha_3$};
\node at (6.58,4)  {$\alpha_4$};
\node at (7.6,2.35)  {$\alpha_5$};
\node at (8.83,3.9)  {$\alpha_6$};
\end{tikzpicture}
\end{center}
     \caption{A branch-symmetric grope $\Sigma$.}
    \label{figure:example-grope}
\end{figure}
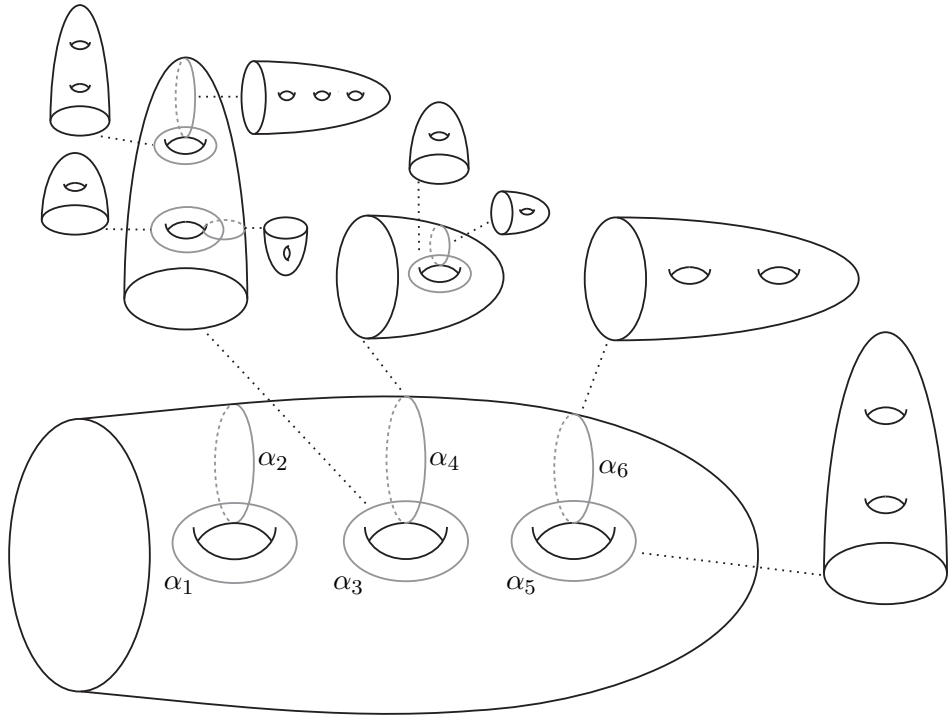

The first stage surface has genus three, so $g_1=3$.  Enumerate the holes/symplectic basis pairs from left to right.  The first pair $\alpha_1,\alpha_2$ has no higher stage surfaces attached to it, so $n_1=0$.  The second pair $\alpha_3,\alpha_4$ has height $2$ symmetric gropes attached, and the third pair $\alpha_5,\alpha_6$ has height one gropes attached.  Thus $n_2=2$ and $n_3=1$.  Next, the second stage surfaces attached to $\alpha_3$ and $\alpha_4$ are of genus two and one respectively.  Therefore $g_2^2=2+1=3$.  The third stage surfaces have the sum of their genera $g^2_3 =(1+1)+(1+1+2+3)=9$.  Finally the second stage surfaces attached to $\alpha_5$ and $\alpha_6$ both have genus two, so $g_2^3=4$.  We can therefore compute the length of $\Sigma$ to be
\[\|\Sigma\|^q = \smfrac{1}{q^0}(1) + \smfrac{1}{q^2}\left(1- \smfrac{1}{3} - \smfrac{1}{9}\right) + \smfrac{1}{q}\left(1- \smfrac{1}{4}\right) = 1+\smfrac{5}{9q^2} + \smfrac{3}{4q}.\]
\end{enumerate}
\end{remark}

\begin{remark}
  There is a natural trade off between the height of gropes versus the genus of the first stage surface.  One can often increase the height of the grope at the expense of increasing the genus of the first or subsequent stages; constructions can be found in, for example, \cite{CT, Hor2}, and our Proposition~\ref{prop:effectinfections}.  For a high $q$ parameter a high grope is valued more, whereas for a low $q$ parameter a low genus of the first stage has more value.  Here ``more value'' means: gives rise to smaller distance.

  We also remark on the operation of grope splitting~\cite[Lemma~4]{krushkal2000exponential}.
  Iterating this operation can make every surface at the second stage or higher of genus one, at the cost of a huge increase in the first stage genus.
An example is shown in Figure~\ref{figure:grope-splitting}. On the left the first stage surface has genus one, one of the second stage surfaces has genus two and the other has genus one.  The length with $q=1$ is $1-\smfrac{1}{3} = \smfrac{2}{3}$. On the right, after grope splitting, the first stage surface is genus two, and there are now four genus one surfaces in the second stage.  Two come from parallel copies of the genus one second stage surface from the original grope, and two come from splitting the genus two second stage surface of the original grope into two genus one second stage surfaces for the new split group, via a tube that adds a handle to the original first stage surface.  The length with $q=1$ is $\big(1 - \smfrac{1}{2}\big) + \big(1 - \smfrac{1}{2}\big) = 1$.  So we have arranged all higher surfaces to be genus one, but it has increased the genus of the first stage and increased the length of the grope.  Setting $q>1$ simply divides both lengths by~$q$.

 \begin{figure}[h]
\begin{center}
\begin{tikzpicture}
\node[anchor=south west,inner sep=0] at (0,0){\includegraphics[scale=0.65]{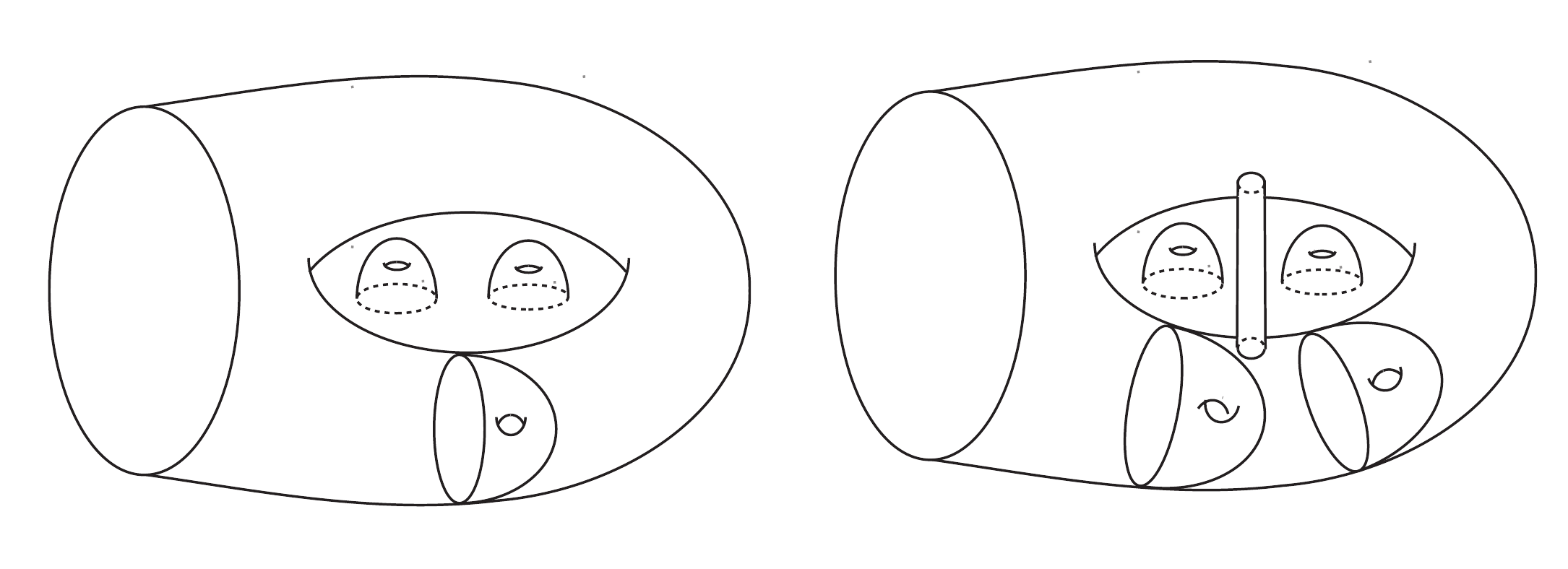}};
\end{tikzpicture}
\end{center}
     \caption{Grope splitting.}
    \label{figure:grope-splitting}
\end{figure}
  \end{remark}

\begin{lemma}
  For every $q \geq 1$, the function $d^q$ determines a pseudo metric on the set of concordance classes of knots $\C$.
\end{lemma}

\begin{proof}
  If two knots are concordant then they cobound an annulus, which is a grope $\Sigma$ with $g_1=0$, and consequently the sum in the equation for $d_{\Sigma}^q(K,J)$ is vacuous.  We note that the function $d_{\Sigma}^q(K,J)$ is always nonnegative, since all the grope attached to the first stage must be symmetric.  It is helpful to consider each hole (i.e.\ dual pair of curves in a symplectic basis) of $\Sigma_{1:1}$ separately, to begin with for $q=1$. For each hole the contribution is $1$ minus the contribution from higher surfaces attached to the two associated dual curves.  On the other hand the highest that each term $1/g^i_k$ can be is $1/2^{k-1}$.  Thus $$\sum_{k=2}^{n_i+1} \frac{1}{g_k^i} \leq \sum_{k=2}^{n_i+1} \frac{1}{2^{k-1}} \leq \sum_{k=1}^{\infty} \frac{1}{2^k} = 1.$$
As a consequence, the infimum of $\|\Sigma\|^1$ over all possible grope concordances between concordant knots $K$ and $J$ is indeed zero.
When $q>1$, the entire contribution of a hole is multiplied by a positive constant $1/q^{n_i}$.  Thus the infimum of $\|\Sigma\|^q$, over all $\Sigma$ bounded by $K$ and $J$, is also zero for any $q >1$.
This shows that the distance function $d_G^q$ is well-defined and that it is a pseudo-metric.

Symmetry is apparent from the definition. The triangle inequality is also straightforward.  To see this suppose $K_0$ and $K_1$ are grope concordant via $\Sigma_{01}$ while $K_1$ and $K_2$ are grope concordant via $\Sigma_{12}$.  Glue the two gropes together along $K_1$ to yield a new grope $\Sigma_{02} = \Sigma_{01} \cup_{K_1} \Sigma_{12}$ in $S^3 \times I$ whose distance function satisfies $\|\Sigma_{02}\|^q = \|\Sigma_{01}\|^q + \|\Sigma_{12}\|^q$.  By considering the infima over all possible gropes we obtain triangle inequality for $d^q$.
\end{proof}

\begin{remark}
It is possible for the $q=1$ distance function to vanish for knots which cobound the simplest possible infinite grope, i.e.\ every surface is genus one.  It is an open question whether such knots are necessarily concordant.
\end{remark}

We call $d^q \colon \C \to \R^{\geq 0}$ the $q$-grope pseudo-metric.  We can also define a pseudo-norm which gives rise to the pseudo-metric.

\begin{definition}\label{defn:grope-norm}
We say that an oriented knot $K$ in $S^3$ \emph{bounds a branch-symmetric grope $\Sigma$} if $\Sigma$ is a branch-symmetric grope, and we have a framed smooth embedding $\iota \colon \Sigma \hookrightarrow B^4$ where $\iota^{-1}(\partial B^4) = \partial \Sigma_{1:1}$ and
$\iota_{| {\iota^{-1}(\partial B^4)}} = K.$
 Define $\|\Sigma\|^q$ as the quantity on the right hand side of Equation~\ref{eq:knot_grope_formula} let
\[ \|K\|^q := \inf_{\Sigma}\{ \|\Sigma\|^q \, | \,  K \text{ bounds a branch-symmetric grope } \Sigma \}.\]
Equivalently, we could define the grope $q$-norm of a knot $K$ to be $\|K\|^q := d^q(K,U)$.
\end{definition}

\noindent We note that, for knots $K$ and $J$,
$$ \|K \# -J\|^q = d^q(K,J).$$

\begin{remark}\label{remark:top-slice-knots}
  While a topologically slice knot bounds an arbitrary height smoothly embedded grope, the genera of the higher stages grow exponentially with the height in the known proofs~\cite[Proposition~2.2.4]{quinn1982-endsofmaps},~\cite[Theorem~5.2]{gompf2005-steinsurfaces}.
  We do not know how to prove, and indeed doubt that it is true, that a topologically slice knot bounds the simplest infinite grope.  Thus it is possible that with the $q=1$ pseudo-metric, topologically slice knots are positive distance from the unknot, meaning that $d^q$ is a metric instead of a pseudo-metric.  On the other hand topologically slice knots have zero distance from the unknot for $q>1$, since they bound gropes with arbitrarily large height, all of which have the same first stage genus.  As mentioned in the introduction, we have also been investigating metrics which, by taking signs into account, connect much more closely with the known obstructions for topologically slice knots to be smoothly slice.
  \end{remark}

\begin{example}
The trefoil and the figure eight knot are not algebraically concordant, therefore they do not cobound any grope of height 3.  We will see shortly that this implies that the distance $d^q(3_1,4_1) \geq 1/2q$.
The purpose of this example is to explain how to compute an upper bound for the distance between the trefoil $3_1$ and the figure eight knot $4_1$.  First, the connect sum is genus two, so~$2$ is certainly an upper bound.
A standard Seifert surface for $K:= 3_1 \# -4_1$ has an unknotted curve of self-linking zero on it, so the slice genus of $K$ is one, and so in fact $d^q(3_1,4_1) \leq 1$.
However the sum $3_1 \# 4_1$ is also Arf invariant zero, and so therefore bounds a grope of height two, as we shall see.  We are not able to extend the genus one surface just constructed to a grope of height two.  Instead we will find a grope of height two based on a genus two first stage surface, by finding second stage surfaces to attach to the pushed in Seifert surface $F$.  For suitably high values of $q$, this will reduce our upper bound.

 A Seifert form for $K$ is given by
 \[A= \begin{pmatrix}
   1 & 0 & 0 &0 \\ 1 & 1 & 0 & 0 \\ 0 & 0 & -1 & 0 \\ 0 & 0 & 1 & 1
 \end{pmatrix} \]
We can change basis using
\[B= \begin{pmatrix}
  -2 & 1 & 1 & -1 \\
  1 & 0 & 1 & 1 \\
  3 & -1 & 0 & 1 \\
  -1 & 1 & 1 & 0
\end{pmatrix}\]
so that
\[B^T A B = \begin{pmatrix}
  -8 & 2 & -1 & -2 \\ 3 & 0 & 2 & 1 \\ -1 & 2 & 4 & 0 \\ -2 & 1 & 1 & 0
\end{pmatrix}\]
Now all diagonal entries have even self intersections.  The corresponding curves bound on $F$ bound immersed disks further into the 4-ball than the Seifert surface~$F$.  We will construct these disks and then improve them into framed embedded surfaces.

Draw a link (with intersections between dual curves) on the Seifert surface for $3_1 \# 4_1$ representing the curves in the new basis.  Push them off the Seifert surface so that they become disjoint.  The self-linking numbers correspond to the difference between the framing of the surfaces' normal bundles and the desired framing in order to have a framed grope.
Since the diagonal entries of $B^TAB$ are even, the failure of second stage surfaces to be correctly framed will be rectified by adding local cusps, which change the framing on the boundary by $\pm 2$.  The resulting self intersections will then be resolved into additional genus.

Find a null-homotopy of the link components, ignoring the Seifert surface, recording each time one component crosses another.  This determines immersed caps for the Seifert surface, further into the 4-ball.  Add cusps to fix the framings as above.  Convert all self-intersections into genus in the standard way, adding a Seifert surface for the Hopf link, a twisted annulus instead of a neighborhood of the intersection point. We then remove the intersections between different surfaces, following \cite[Theorem~8.13]{COT}.  There is a linking torus $T_{\alpha_i}$ in a neighborhood of each basis curve $\alpha_i \subset F$, the boundary of a regular neighborhood $\partial (\cl(\nu \alpha_i))$, which is disjoint from the surface attached to $\alpha_i$, is disjoint from $F$, but intersects the surface $\Sigma_{\beta_i}$ attached to $\beta_i$, where $\beta_i$ is the basis curve on $F$ (in the new basis) that is symplectically dual to $\alpha_i$.  Tube each intersection of another surface into (a parallel copy of) $T_{\alpha_i}$.  This removes the intersection at the cost of increasing the genus.  Repeat this operation to obtain disjointly embedded framed second stage surfaces.

It thus remains to determine upper bounds for the second stage surface genera.  If we show the sum of the genera to be at most $(k_1,k_2)$ for the first and second symplectic basis pair respectively, then $$\|K\|^q \leq \frac{1}{q} \Big(2-\frac{1}{k_1} - \frac{1}{k_2} \Big).$$
We found a height 2 grope with $k_1=16$, $k_2=4$, which gives an upper bound of $27/16q$. For $q$ sufficiently high, this upper bound improves on the genus one surface realizing the slice genus constructed at the beginning of this example.
Note that for a better upper bound, if a certain amount of second stage genus seems inevitable, one should arrange for as much of the genus as possible to appear above one basis pair.
\end{example}

While the height two grope constructed above starts with a pushed in Seifert surface, we remark that all known constructions of higher gropes do not proceed in this way.  Instead they construct high gropes at the expense of increasing first stage genus.  If this difficulty could be circumvented, we could prove the smooth case of Theorem~\ref{cor:notdiscrete} with $q=1$.

We can extend our metrics to links which have \emph{all pairwise linking numbers vanishing}.  We say that two ordered, oriented links $K$ and $J$ \emph{cobound a branch-symmetric grope $\Sigma$} if there
is a framed smooth embedding of $\iota \colon \Sigma = \sqcup_j \Sigma_j  \hookrightarrow S^3 \times I$
where $\Sigma_j$ is  a branch-symmetric grope for $1 \leq j \leq m$, with
$\iota^{-1}(\partial (S^3 \times I)) = \partial \Sigma_{1:1}$, $\iota_{| {\iota^{-1}(S^3 \times \{0\})}} = K$ and $\iota_{| {\iota^{-1}(S^3 \times \{1\})}} = -J$.  We also say that $K$ and $J$ are \emph{grope concordant via $\Sigma$}.
Then define
$$\|\Sigma\|^q := \sum_{j=1}^m\sum_{i=1}^{g_1(\Sigma_j)}\frac{1}{q^{n_{ji}(\Sigma)}} \bigg( 1- \sum_{k=2}^{n_{ji}(\Sigma)+1} \frac{1}{g_k^{ji}(\Sigma)}\bigg)$$
where $g_k^{ji}(\Sigma) := g^{i}_k(\Sigma_j)$ and $n_{ji}(\Sigma)= n_i(\Sigma_j)$.  Now we can define
\[d^q(K,J) = \inf_{\Sigma}\{\|\Sigma\|^q \, |\, K\text{ and }J \text{ cobound a branch-symmetric grope } \Sigma \text{ in } S^3 \times I\}.\]
Similarly we can define the notion of an $m$-component link $L$, with vanishing pairwise linking numbers, \emph{bounding a branch-symmetric grope}, just like we did for knots, and use this to define the $q$-norm $\|L\|^q$ of the link~$L$.   It is easy to see that $\|L\|^q = d^q(L,U)$, where $U$ is the $m$-component unlink.

Recall that an ordered, oriented, $m$-component link $L$ \emph{bounds a symmetric grope $G$ of height $n$} if $G = G^1 \sqcup \dots \sqcup G^m$ is a disjoint union of height~$n$ symmetric gropes~$G^i$ for $i=1,\dots, m$, and we have a framed smooth embedding $\iota\colon G \hookrightarrow B^4$ where $\iota^{-1}(\partial B^4) = \partial G$ and $\iota_{| \partial G_i} = L_i$ for all~$i$.  Here $\partial G_i$ is the boundary of the first stage surface of $G_i$ and $\partial G = \sqcup_i \partial G_i$.
In this case, we say that $L \in \mathcal{G}_{n}$.

\begin{proposition}\label{prop:boundedbelow} Let $L$ be a link and $n\geq 2$.   If $L$ bounds a branch-symmetric grope $\Sigma$ that has a branch of length at most $n-1$ then $\|\Sigma\|^q \geq \frac{1}{q^{n-2}2^{(n-2)}}$.
Thus, if $L\notin \mathcal{G}_{n}$, then $\|L\|^q\geq \frac{1}{q^{n-2}2^{(n-2)}}$.
\end{proposition}

\begin{proof}
   Let $\Sigma=(\Sigma_1,\dots,\Sigma_m)$ be a branch-symmetric grope bounding $L$ with a branch that has length at most $n-1$.  Without loss of generality assume that it is the first branch of $\Sigma_1$, so $n_{11}\leq n-2$. By Lemma~\ref{lemma:counting-genera}, $g_k^{ji}\geq 2^{k-1}$.  Therefore
\begin{align*}
  \|\Sigma\|^q\geq \|\Sigma_1\|^q &=\sum_{i=1}^{g_1(\Sigma_1)}\frac{1}{q^{n_{11}}} \left( 1- \sum_{k=2}^{n_{1i}+1} \frac{1}{g_k^{1i}}\right)\geq \frac{1}{q^{n_{11}}} \left( 1- \sum_{k=2}^{n_{11}+1} \frac{1}{2^{k-1}}\right)\\
&=\frac{1}{q^{n_{11}}2^{n_{11}}}\geq \frac{1}{q^{n-2}2^{(n-2)}} = \frac{1}{(2q)^{n-2}}.
\end{align*}
Thus $\|L\|^q\geq \frac{1}{q^{n-2}2^{(n-2)}}$.
\end{proof}

Obstructions to a knot or link being $(n+2)$-solvable are obstructions to the knot or link bounding a grope of height $n$, so the preceding proposition translates many results from the literature on the solvable filtration into statements about the distance between knots in our grope metric.
For the convenience of the reader we recall the definition of $n$-solvability for knots, originating from \cite{COT}, and reformulated as given below in~\cite[Definition~2.3]{CHL5}.

\begin{definition}\label{defn:n-solvable}
 We say that a knot $K$ is $(n)$-solvable if the zero surgery manifold $M_K$ bounds a compact oriented 4-manifold $W$ with the inclusion induced map $H_i(M_K;\Z) \to H_i(W;\Z)$ an isomorphism for $i=0,1$, and such that $H_2(W;\Z)$ has a basis consisting of $2k$ embedded, connected, compact, oriented surfaces $L_1,\dots ,L_k,D_1,\dots,D_k$ with trivial normal bundles satisfying:
\begin{enumerate}[(i)]
\item    $\pi_1(L_i) \subset \pi_1(W)^{(n)}$ and $\pi_1(D_j) \subset \pi_1(W)^{(n)}$ for all $i,j=1,\dots,k$;
\item the geometric intersection numbers are $L_i \cdot L_j = 0 = D_i\cdot D_j$ and $L_i \cdot D_j = \delta_{ij}$ for all $i,j =1,\dots,k$.
\end{enumerate}

The subgroup of $\C$ of $(n)$-solvable knots is denoted $\mathcal{F}_{(n)}$.  Such a $4$-manifold $W$ is called an $n$-solution.
It is not too hard to see that an $(n)$-solvable knot in the sense above is $(n)$-solvable in the sense of~\cite{COT}.

\end{definition}

\noindent The following was proved in \cite{COT}.

\begin{theorem}[Theorem~8.11 of \cite{COT}.]\label{thm:height-n-plus-two-grope-implies-n-solvable}
  For all $n \geq 0$, \[ \mathcal{G}_{n+2} \subseteq \mathcal{F}_{(n)}.\]
\end{theorem}

Next we state our main theorem, and give the proof modulo Proposition~\ref{ex:specificdecreasing}, which actually constructs the sequence of knots with decreasing but nonzero norms.

\begin{theorem}\label{cor:notdiscrete}
 For any $q>1$ there exist uncountably many sequences of knots $\{K_i\}_{i \geq 0}$ such that $d^q(K_i,U) > 0$ for all $i$ but $d^q(K_i,U) \to 0$ as $i \to \infty$.
In particular, if $q>1$ then neither of the topologies on $\C$ and $\C^{top}$, induced by $d^q$ and $d^q_{top}$ respectively, are discrete.
\end{theorem}

As remarked in the introduction, there is a straightforward proof for the second sentence in the smooth case, due to the fact that our $q>1$ pseudo-metrics on $\C$ are not metrics.

\begin{proof} In Proposition~\ref{ex:specificdecreasing} we will exhibit a family of knots $K_n$ for each $n \geq 0$ that satisfies the following hypotheses.
The knot $K_n$ bounds a symmetric height $n+2$ grope, whose first stage has genus $2^{n}$ and whose higher stages have genus one. Suppose also that $K_n\notin \mathcal{G}_{n+3}$. Then these knots represent distinct concordance classes and
$$
 \frac{1}{q^{(n+1)}2^{(n+1)}}\leq \|K_n\|^q\leq \frac{1}{2q^{n+1}},
 $$
 where the lower bound is from Proposition~\ref{prop:boundedbelow} and the upper bound is by virtue of the hypothesized height $n+2$ grope.
 Therefore, for any $q>1$, $K_n$ converges to the class of the trivial knot since $\|K_n\|^q\rightarrow 0$.  We will also show in the proof of Proposition~\ref{ex:specificdecreasing} how to modify the construction to obtain infinitely many different such sequences of knots.
This completes the proof of Theorem~\ref{cor:notdiscrete} modulo the rather large caveat of Proposition~\ref{ex:specificdecreasing}.
\end{proof}

\section{Relations to knot signatures}\label{section:relations-to-knot-signatures}

\begin{theorem}\label{thm:refinedslicegenusbounds}
Suppose that $K$ bounds a branch-symmetric grope $\Sigma$ in $B^4$. Let $g$ be the number of branches of $\Sigma$ for which the branch length $n_i$ is less than $2$. Then the Levine-Tristram signature $\sigma_K(z)$ satisfies
\begin{equation}\label{eq:refinedslicegenusbound}
|\sigma_K(z)|\leq 2g,
\end{equation}
 for any complex number $z$ of norm one that is not a root of the Alexander polynomial $\Delta_K$.
\end{theorem}

\begin{proof}
From the branch-symmetric grope $\Sigma$, we will construct a simply connected $4$-manifold $V$, with $\partial V\cong S^3$, in which $K$ bounds a (null-homologous) slice disk $\Delta$. Recall that, by definition, $\Sigma$ comes equipped with an identification of its tubular neighborhood with the product $\Sigma\times D^2$.  For each branch of $\Sigma$ for which $n_i$ is less than $2$, perform surgery on $B^4$ along a push off $\alpha_i'$ of \textit{one} of the two circles, $\alpha_i$, out of the two dual attaching circles for that branch. For each branch of $\Sigma$ for which $n_i$ is at least $2$, perform surgery on $B^4$ along circles, $\alpha'_i, \beta'_i$, which are push-offs of \textit{both} of the two dual attaching circles $\alpha_i, \beta_i$ for that branch. Here $\alpha'_i, \beta'_i$ are push-offs of $\alpha_i,\beta_i$ respectively, into the second stage surfaces of grope that they bound. The framings for these surgeries are dictated by the product structure on the tubular neighborhoods.

Since any link in the interior of $B^4$ bounds a collection of disjointly embedded disks, the result of such surgeries is a  manifold $V$ that is diffeomorphic to a punctured connected sum of copies of either $S^2\times S^2$ or $S^2\tilde{\times}S^2$. In particular $\sigma(V)=0$ and $H_2(V)$ is free abelian of rank $2g+4e$, where  $e$ is the number of branches of $\Sigma$ for which $n_i$ is at least $2$. Since, after surgery, the corresponding circles bound embedded disks whose interiors are disjoint from $\Sigma$, the slice disk $\Delta$ is essentially  $\Sigma$ surgered ambiently using two copies of each of these disks. It follows that $\Delta$ is null homologous and $H_2(V-\Delta;\Z)\cong H_2(V;\Z)$. In the cases $\alpha'_i, \beta'_i$, the process of forming $\Delta$ is called \textit{symmetric surgery}, and is described in the proof of~\cite[Theorem 8.11,~$h=1$]{COT} (see also \cite[Section~2.3]{FQ}). Moreover, in the proof of \cite[Theorem 8.11]{COT}, a very precise collection of oriented surfaces, $\mathcal{E}=\{S_j, B_j|~ 1\leq j\leq 2e\}$ was described, representing a basis for the second homology of the $2e$ copies of $S^2\times S^2$ created by the surgeries on the $2e$ circles $\alpha'_j, \beta'_j$.
 We also use the procedure of the $h=1.5$  part of the proof, which tubes $S_{i+g}$ twice into parallel copies of the $B_i$, in order to remove intersections between $S_i$ and $S_{i+g}$ that arose from pushing off the contraction.

These surfaces have the following properties. They are disjointly embedded in $V-\Delta$ except that $S_j$ intersects $B_j$ transversely once with positive sign, and they have trivial normal bundles. Moreover the $B_j$ are essentially the capped-off second stage grope surfaces, which, since third stage grope surfaces exist, satisfy  $\pi_1(B_j)\subset \pi_1(V-\Delta)^{(1)}$. The $S_j$ begin as $2$-spheres pushed off the contraction, then half of them are tubed into copies of the $B_j$ to remove intersections created by the push off operation.  All the $S_j$ also satisfy $\pi_1(S_j)\subset \pi_1(V-\Delta)^{(1)}$.  Consequently, not only is the intersection matrix for the corresponding summand of $H_2(V-\Delta;\Z)$ a direct sum of $2e$ hyperbolic matrices, but this matrix even represents the intersection form for $H_2(V-\Delta;\Z[t,t^{-1}])$ (for this summand).

Now we follow the proofs of ~\cite[Proposition 4.1]{Cochran-Harvey-Horn:2013-1} and~\cite[Theorem 3.7]{CLick}. Let $d=p^r$ be a prime power and let $\Sigma_d(K)$ denote the $d$-fold cyclic cover of $S^3$ branched over $K$, which is well known to be a $\Z_p$-homology sphere ~\cite[Lemma 4.2]{CG1}. Since $\Delta$  represents zero in $H_2(V,\partial V)$, $H_1(V-\Delta)\cong\Z$, generated by the meridian. Thus the $d$-fold cyclic cover of $V$ branched over $\Delta$, denoted  $\widetilde{V}$, is defined and has boundary $\Sigma_d(K)$. Since $H_1(V;\Z_p)=0$, it follows from the proof of ~\cite[Lemma 4.2]{CG1} that $H_1(\widetilde{V};\Z_p)=0$. Thus the first and third betti numbers vanish: $\beta_1(\widetilde{V};\Z_p)=0=\beta_3(\widetilde{V};\Z_p)$.

To compute the signature of $\widetilde{V}$ we make $V$ into a closed $4$-manifold and use the $G$-signature theorem. Let $(B^4,F_K)$ be the $4$-ball together with a Seifert surface for $K$ pushed into its interior. Let
$$
(Y,F)=(V,\Delta)\cup (-B^4,-F_K)
$$
be the closed pair, let $\widetilde{W}$ denote the $d$-fold cyclic branched cover of $(B^4,F_K)$, and let $\widetilde{Y}$ be the $d$-fold cyclic branched cover of $(Y,F)$. Note that $\Z_d$ acts on $\widetilde{V}$, $\widetilde{Y}$, and $\widetilde{W}$ with $V$, $Y$ and $B^4$ respectively as quotient. Choose a generator $\tau$ for this action. Let $H_i(\widetilde{Y},j;\mathbb{C})$, $0\leq j<d$, denote the $\exp(\frac{2\pi i j}{d})$-eigenspace for the action of  $\tau_*$ on $H_i(\widetilde{Y};\mathbb{C})$; let $\beta_i(\widetilde{Y},j)$ denote the rank of this eigenspace, and let $\chi(\widetilde{Y},j)$ denote the alternating sum of these ranks (similarly for $\widetilde{V}$ and $\widetilde{W}$). Let $\sigma(\widetilde{Y},j)$ denote the signature of the $\exp(\frac{2\pi i j}{d})$-eigenspace of the isometry $\tau_*$ acting on $H_2(\widetilde{Y};\mathbb{C})$ (similarly for $\widetilde{V}$ and $\widetilde{W}$). By a lemma of Rochlin, using the $G$-signature theorem ~\cite{Rok1}\cite[Lemma 2.1]{CG1}, since $\widetilde{Y}$ is closed and $[F]\cdot[F]=0$,
$$
\sigma(\widetilde{Y},j)=\sigma(Y).
$$
Since $\widetilde{Y}=\widetilde{V}\cup -\widetilde{W}$ glued along the rational homology sphere $\Sigma_d(K)$, this translates to
$$
\sigma(\widetilde{V},j)-\sigma(\widetilde{W},j)=\sigma(V)-\sigma(B^4)=0.
$$
Moreover it is known that $\sigma(\widetilde{W},j)$ is a \emph{$p^r$-signature of $K$}~\cite{Vi1}, \cite[Chapter 12]{Gor1}, implying that
\begin{equation}\label{eq:sigs1}
\sigma_K(\omega^j)=\sigma(\widetilde{W},j)=\sigma(\widetilde{V},j)
\end{equation}
where $\omega=\exp(\frac{2\pi i}{d})$.
Since these roots of unity are dense in the circle, $\sigma_K(z)=\sigma_K(\omega^j)$ for some $r$ and $j$. Hence it suffices to show that
\begin{equation}\label{eq:sigs2}
|\sigma(\widetilde{V},j)|\leq 2g.
\end{equation}

Since $\pi_1(\mathcal{E})\subset \pi_1(V-\Delta)^{(1)}$, the collection $\mathcal{E}$ lifts to a collection
$$
\widetilde{\mathcal{E}}=\{t^k\widetilde{S}_i, t^j\widetilde{B}_i ~|~ 0\leq k,j <d-1, 1\leq i\leq 2e\}
$$
of $4ed$ embedded surfaces  in $\widetilde{V}$. Indeed, their regular neighborhoods lift, so each has self-intersection $0$. Moreover, for different $i$ and $k$ the collections of lifts are disjoint except that  $t^k\widetilde{S}_i$ and $t^k\widetilde{B}_i$ intersect transversely in one point.   Since duals exist, it can easily be seen that $\widetilde{\mathcal{E}}$ is a basis for a $\tau$-invariant subspace  of $H_2(\widetilde{V};\mathbb{C})$ of rank $4ed$ (for more details see the proof of~\cite[Theorem 6.2]{Cochran-Harvey-Horn:2013-1}). Hence we can speak of the $j^{th}$ eigenspace $(\widetilde{\mathcal{E}},j)$. For each fixed $i$, let $\widetilde{\mathcal{S}}_i$ and $\widetilde{\mathcal{B}}_i$ denote the $\tau$-invariant $d$-dimensional subspaces with bases $\tau^k\widetilde{S}_i$ and $\tau^k\widetilde{B}_i$ as $k$ varies. Since the roots of $t^d-1$ are distinct, the $j^{th}$ eigenspaces of $\widetilde{\mathcal{S}}_i$ and $\widetilde{\mathcal{B}}_i$  have dimension one, generated by, say, $s_{ij}$ and $b_{ij}$ where $s_{is}\in \widetilde{\mathcal{S}}_i$. Thus $(\widetilde{\mathcal{E}},j)$ has dimension $4e$ and has $\{s_{ij}|~1\leq i\leq 2e\}$ and $\{b_{ij}|~1\leq i\leq 2e\}$ generating Lagrangian subspaces of rank $2e$. Hence $\sigma(\widetilde{\mathcal{E}},j)=0$.

Finally note that the $j^{th}$-eigenspace of $H_2(\widetilde{V};\mathbb{C})$ decomposes as $(\widetilde{\mathcal{E}},j)\oplus D_j$  for some $D_j$  (the direct sum is orthogonal with respect to the intersection form since the intersection form restricted to  $(\widetilde{\mathcal{E}},j)$ is nonsingular).
Hence the rank of $D_j$ is $\beta_2(\widetilde{V},j)-4e$. But the argument on the top of page 2118 of ~\cite{Cochran-Harvey-Horn:2013-1}, in particular Equation 4.4, establishes that $\beta_2(\widetilde{V},j)=\beta_2(V)$. Since the latter is $4e+2g$, the rank of $D_j$ is $2g$. Thus
$$
|\sigma(\widetilde{V},j)|=|\sigma(D_j)|\leq \text{rank}D_j=2g,
$$
establishing ~\eqref{eq:sigs2} and finishing the proof.
\end{proof}

\begin{corollary}\label{cor:signatureboundsnorm}  If $K$ is a knot and $z$ is a complex number of norm one that is not a root of $\Delta_K$ then
$$
\frac{|\sigma_K(z)|}{4q}\leq \|K\|^q.
$$
Moreover if $\Arf(K) \neq 0$, then
$$
 1 + \frac{\max\{|\sigma_K(z)|-2,0\}}{4q}\leq \|K\|^q.
$$
\end{corollary}

\begin{proof}  Suppose $K$ bounds a branch-symmetric grope $\Sigma$ in $B^4$. Let $b_0, b_1, e$ denote the number of branches of $\Sigma$ for which $n_i$ is $0,1$, or greater than $1$, respectively.  We order the branches so that $n_i=0$ for $1 \leq i \leq b_0$, $n_i=1$ for $b_0 +1 \leq i \leq b_0 + b_1$, and $n_i\geq 2$ for $i \geq b_0 + b_1 +1$.   Then
$$
\|\Sigma\|^q=b_0+\sum_{i=b_0+1}^{b_0+b_1}\frac{1}{q} \bigg( 1- \frac{1}{g_2^{i}}\bigg)+\sum_{i=b_0+b_1+1}^{g_1(\Sigma)}\frac{1}{q^{n_{i}}} \bigg( 1- \sum_{k=2}^{n_{i}+1} \frac{1}{g_k^{i}}\bigg)\geq b_0+\frac{b_1}{2q}\geq \frac{g(\Sigma)}{2q},
$$
where $g(\Sigma)=b_0+b_1$  
By Theorem~\ref{thm:refinedslicegenusbounds}, $2g(\Sigma) \geq |\sigma_K(z)|$. Thus
$$
\frac{|\sigma_K(z)|}{4q}\leq \|\Sigma\|^q.
$$
Since this true for every $\Sigma$, the first claimed result follows.
For the second part, if $\Arf(K)\neq 0$ then $K$ cannot bound any symmetric grope of height 2, so $b_0 \geq 1$.  Thus
$$b_0 +b_1/2q \geq 1 + \frac{g(\Sigma)-1}{2q} \geq 1 + \frac{\max\{|\sigma_K(z)|-2,0\}}{4q},$$
since $g^i_k \geq 2$,
from which the second claimed result follows.
\end{proof}

\section{Definitions on string links}\label{section:defns-string-links}

\noindent Let $I$ denote the interval $[0,1]$.

\begin{definition}[String links and string link concordance] \label{def:string_link}
  Fix $m$ points, $p_1,\dots,p_m \in D^2$.  An $m$-component \emph{string link} $L$ is an embedding $L \colon \{p_1,\dots,p_m\} \times I \hookrightarrow D^2 \times I$ such that $(p_i,j) \mapsto (p_i,j)$ for $i=1,\dots m$ and $j=0,1$.  An example is depicted in Figure~\ref{figure:string-link}.
  Let $L^i = \im(\{p_i\} \times I)$ be the $i^{th}$ component of $L$.  Denote the exterior of a string link $L$ by $E_L := D^2 \times I \sm \nu L$, where we identify $L$ with its image, and $\nu L$ is a regular neighborhood of $L$.

\begin{figure}[h]
\begin{center}
\begin{tikzpicture}
\node[anchor=south west,inner sep=0] at (0,0){\includegraphics[scale=0.3]{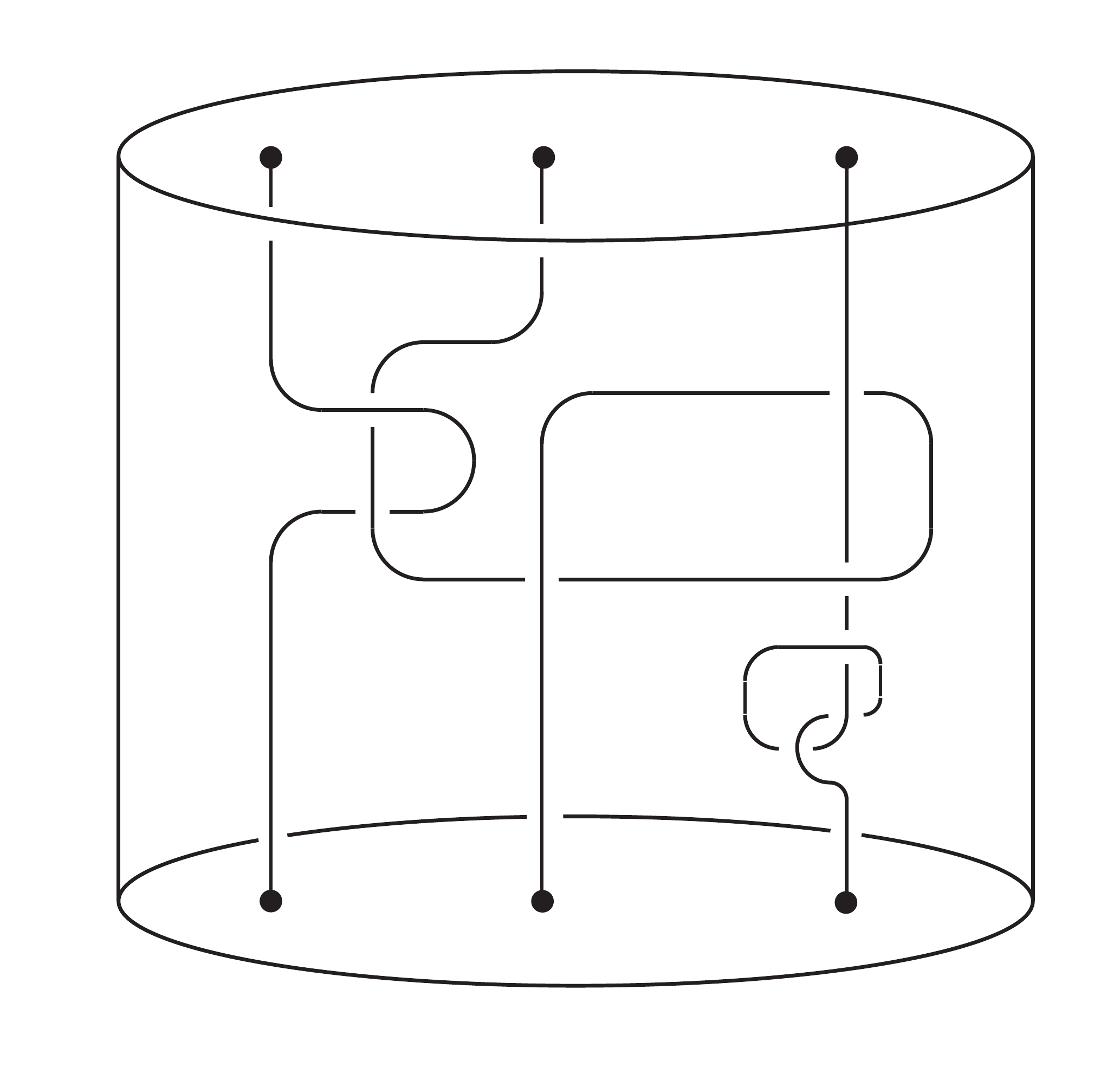}};
\end{tikzpicture}
\end{center}
     \caption{A $3$-component string link.}
    \label{figure:string-link}
\end{figure}

A \emph{concordance} between string links $L_0,L_1$  is an embedding $\{p_1,\dots,p_m\} \times I \times I \subset D^2 \times I \times I$ with $\im(\{p_i\} \times I \times \{j\}) = L_j^i \subset D^2 \times I \times \{j\}$ for $i=1,\dots,m$ and $j=0,1$, and $(p_i,k,x) \mapsto (p_i,k,x)$ for $i=1,\dots,m$, $k=0,1$ and for all $ x \in I$.  We say that $L_0,L_1$ are \emph{string link concordant} or \emph{concordant}.
\end{definition}

Note that a string link $L_0$ is concordant to the trivial string link, $L_T \colon (p_i,x) \mapsto (p_i,x)$ for all $x \in I$, if and only if its closure \[\widehat{L_0} := L_0 \colon \{p_1,\dots,p_m\} \times I \subset \frac{D^2 \times I}{\{(x,0) \sim (x,1) |\,\, x \in D^2\}} \cup_{S^1 \times S^1} S^1 \times D^2 \cong S^3\]
is slice.  We call the string link $L_0$ slice too.

The orientation of $I$ (in $\{p_1,\dots,p_m\} \times I$) determines an orientation of $L$.  We often conflate $L$ and its oriented image, and use $L$ to denote both.

\begin{definition}[Concordance group of string links]
The sum $L_0 \# L_1$ of two string links is given by
\[ \ba{rcl} L_0 \# L_1 \colon \{p_1,\dots,p_m\} \times I &\hookrightarrow  & D^2 \times I \\
(p_i,x) & \mapsto & \begin{cases}
  L_0(p_i,2x) & 0 \leq x \leq 1/2 \\
  L_1(p_i,2x-1) & 1/2 < x \leq 1
\end{cases} \ea\]
The inverse $-L$ of a string link $L$ is given by
\[\ba{rcl} -L \colon \{p_1,\dots,p_m\} \times I &\hookrightarrow  & D^2 \times I \\
(p_i,x) & \mapsto & (p_i,-x)
\ea\]
With these notions of addition and inverse, the set of concordance classes of $m$-component string links form a group $\mathcal{C}^m_{SL}$, the \emph{string link concordance group}.
\end{definition}

Now we give the definition of a string link infection.  One should think about these as operators (functions) $\C^m_{SL} \to \C^\ell$.

\begin{definition}[String link infection]\label{def:stringlinkinfection} An \emph{$m$-multidisk} $\mathbb{D}$ is the standardly oriented unit disk $D^2$ together with a collection of $m$ ordered embedded subdisks $D_1, \dots, D_m$ in $D^2$ with $p_i \in \Int(D_i)$.  Here the $p_i$ are the same points as in Definition~\ref{def:string_link}.

Let $R = R_1 \cup \cdots \cup R_{\ell} \subset S^3$ be a link, let $\mathbb{D}$ be an $m$-multidisk,  and let $\psi: \mathbb{D} \hookrightarrow S^3$ be an  embedding where $R$ intersects $\mathbb{D}$ transversely and $R \cap (\mathbb{D} \smallsetminus (\cup_i \Int(D_i))) = \emptyset$.
The data $(R,\psi)$ is called a \emph{pattern}.  Two patterns are equivalent if they are ambiently isotopic through patterns.  An example of thickened multidisk together with some strands of $R$ is shown in Figure~\ref{figure:multi-disc}.

\begin{figure}[h]
\begin{center}
\begin{tikzpicture}
\node[anchor=south west,inner sep=0] at (0,0){\includegraphics[scale=0.3]{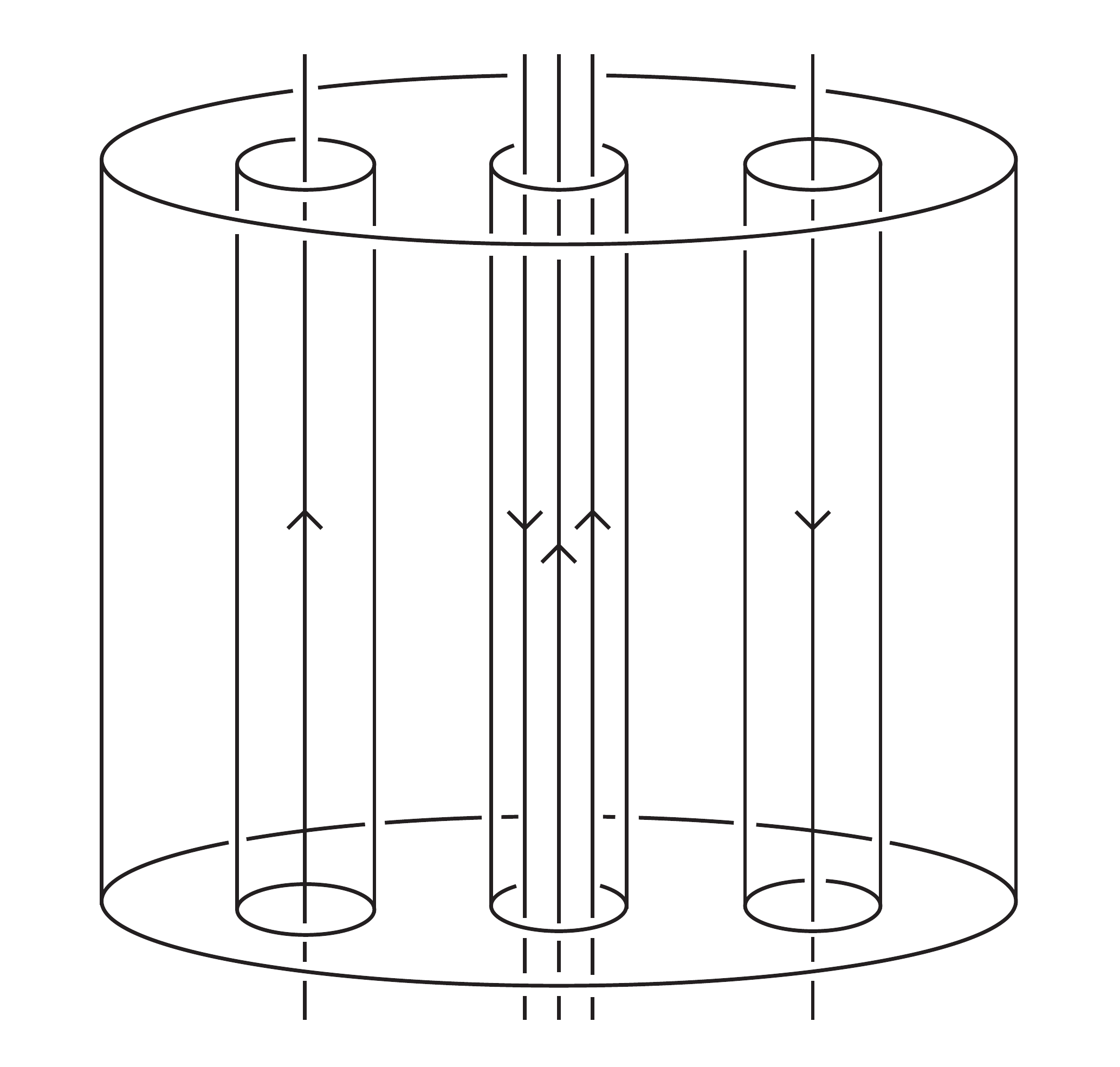}};
\end{tikzpicture}
\end{center}
     \caption{A $3$-multidisk and some strands of $R$ that intersect it.}
    \label{figure:multi-disc}
\end{figure}

This data determines a \emph{satellite operator} $R(-,\psi) \colon \mathcal{C}^m_{SL} \to \mathcal{C}^{\ell}$ from the concordance group of $m$-component string links to the set of concordance classes of $\ell$-component links, as follows.

Set $\mathbb{B} = \mathbb{D} \times I$, oriented using the product orientation, and set $\mathbb{H} = \mathbb{B} \smallsetminus (\cup_i \Int(D_i) \times I) \subset \mathbb{B}$.  Denote the boundary of $\psi(D_i)$ by $\eta_i(R,\psi)$, and let $\eta(R,\psi)$ be the $m$-component unlink $\eta_1(R,\psi) \cup \cdots \cup \eta_m(R,\psi)$.  When it is clear, we will suppress the $(R,\psi)$ and just write $\eta$ or $\eta_i$.  We identify $\mathbb{D}$ with $\mathbb{D} \times \{0\}$ and note that
$\psi$ extends to an orientation preserving embedding $\psi: \mathbb{B} \hookrightarrow S^3$ with  $R \cap \mathbb{H} = \emptyset$.
  Note that the image of $\mathbb{H}$ can be identified with the exterior of the trivial string link $T$, and in this identification $\eta_i$ is the $i^{th}$ meridian of $T$.  Given an $m$-component string link $L$, remove the image of $\mathbb{H}$ from $S^3$, and replace it by $E_L$, identifying the $i^{th}$ longitude of $L$ with the $i^{th}$ longitude of $T$ and the $i^{th}$ meridian of $L$ with the $i^{th}$ meridian of $T$.  The resulting 3-manifold is again homeomorphic to $S^3$ (see Definition 2.2 of \cite{CFT} for more details)
\[f \colon \cl(S^3 \sm \psi(\mathbb{H})) \cup E_L \xrightarrow{\cong} S^3.\]
Denote the image $f(R)$ by  $R(L,\psi)$, the output of the string link operator $R(-,\psi)$ acting on $L$.   When it is clear, we may drop the $\psi$ and write $R(L)$.  In addition, when $L$ has one component, $\psi$ is determined by the curve $\eta$ so we may write $R(L,\eta)$.

By definition, the \emph{(algebraic) winding matrix} of $R(-,\psi)$ is an $\ell \times m$ matrix over $\Z$ with columns $aw_1,\dots,aw_m$, where each $aw_i \in \Z^\ell$ is the element of $H_1(E_R;\Z) \toiso \Z^\ell$ represented by $\eta_i(R,\psi)$.
We say that an operator has (algebraic) \emph{winding number $k$} if every entry of the matrix is equal to~$k$.

The \emph{geometric winding} number of $(R,\psi)$ is an $m$-tuple $(w_1,\dots,w_m) \in (\mathbb{N}_0)^m$ where $w_i$ is the minimal number of intersections of $R$ with the subdisk $D_i$, i.e.\ the number of strands of $R$ that pass through $\eta_i(R,\psi)$.  Here the minimum is taken over all representatives of the pattern equivalence class of $(R,\psi)$.  We remark that for the geometric winding number, the count does not take orientations into account.
\end{definition}

The above definition can be easily adapted to the case that $R$ is a string link.  In that case we obtain a function
  \[R(-,\eta) \colon \mathcal{C}^m_{SL} \to \mathcal{C}^\ell_{SL}.\]
  Next we define the notion of grope concordance of string links.

\begin{definition}
  A \emph{grope concordance} between $m$-component string links $L_0$ and $L_1$ is an embedding of an $m$-component grope $G \subset D^2 \times I \times I$ with $\partial G \subset \partial(D^2 \times I \times I)$, $\partial G \cap (D^2 \times I \times \{j\}) = L_j$ and $\partial G \cap (D^2 \times \{k\} \times I) = \{p_1,\dots,p_m\} \times \{k\} \times I$, for $k,j \in \{0,1\}$.  We say that $L_0$ and $L_1$ cobound the grope~$G$.
\end{definition}


We can easily extend our metrics to string links.

\begin{definition}
  The distance $d^q(L,J)$ between two $m$-component string links $L$ and $J$ with all pairwise linking numbers vanishing, is the distance $d^q(\widehat{L},\widehat{J})$ between their closures $\widehat{L}$ and $\widehat{J}$.  We say that a string link $J$ bounds a group of height $n$, that is $J \in \mathcal{G}_n^m$, if the closure satisfies $\widehat{J} \in \mathcal{G}_n^m$.
\end{definition}

\section{The effect of satellite operations and string link infections}\label{sec:infections}

Classical satellite operators are operators on the metric spaces $(\mathcal{C},d^q)$. More generally, as described above, string link infections may be viewed as functions $(\C^m_{SL},d^q)\to(\C^\ell,d^q)$. In this section we show that these functions are Lipschitz continuous, where the Lipschitz constant depends on the geometric winding number.  We also show that algebraic winding number zero operators are contraction mappings for any $q$ bigger than the geometric winding number.


\begin{definition}
  A \emph{tip} of a grope $\Sigma$ consists of a basis curve on a top stage surface; that is, a surface in $\Sigma$ to which no further surfaces are attached.

  A \emph{cap} for a grope $\Sigma$ (of multiplicity $k$) is a planar surface $D^2 \sm (\sqcup_{i=1}^k \Int(D_i))$ embedded in $D^4 \sm \nu \Sigma$ such that  $D_i \subset D^2$ is a disk, $\partial D^2$ is a normal framing push off of a tip of $\Sigma$, and the interior boundary $\sqcup_{i=1}^k \partial D_i$ is a collection of meridians of $L = \partial \Sigma_{1:1}$.  Note that this definition of a cap is not standard.  We call each interior boundary component $\partial D_i$ a \emph{tip of the cap} or a \emph{cap tip}.  A \emph{capped grope} is a grope for which every top stage surface has a symplectic basis of tips to which caps are attached.

\end{definition}

 We can similarly define a tip and a cap for a grope concordance, and obtain the notion of a capped grope concordance.

\begin{lemma}\label{lem:gropecaps}
If a link $L$ bounds a grope $G$ in a simply-connected $4$-manifold $W$, then the grope can be capped.
\end{lemma}

\begin{proof}
Consider the embedded framed link $\{\ell_j\}$ consisting of the tip circles of $G$. Any link in the interior of a simply-connected $4$-manifold bounds a set of smoothly immersed $2$-disks $\{\delta_j\}$ in the interior of $W$, which we may assume to intersect $G$ transversely. By boundary twisting, we may also assume that for each immersion the given framing on $\{\ell_j\}$ extends to  $\{\delta_j\}$~\cite[Corollary 1.3B]{FQ}. Any such collection of framed immersions can be replaced by a collection that is disjointly embedded (by pushing intersections off the boundary) ~\cite[Section 1.5]{FQ}.   Then each transverse intersection of one of the embedded disks $\delta_j$ with one of the surface stages of  $G$ can be ``pushed down'' (by an isotopy) to create two new intersections with a lower stage~\cite[Section 2.5]{FQ}. In this fashion we may assume that all of the intersections of $\{\delta_j\}$ with $G$ are with the first stage surfaces. After removing small $2$-disks from the $\delta_j$ at these intersection points, we have a collection of disjointly embedded genus zero surfaces whose boundaries are disjoint copies of circle fibers of the regular neighborhood of the first stage surfaces. These can be joined by disjoint tubes in this circle bundle to meridians of $L$, until we arrive at a collection of caps, that is a disjoint collection of framed genus zero surfaces $F_j$ whose interiors are in the complement of $G$, wherein the boundary of $F_j$ is the tip $\ell_j$ together with a number of disjoint parallel copies of certain meridians of $L$ (the cap tips).
\end{proof}

A similar lemma holds holds when $G$ is a grope concordance.  The next proposition is our main technical result for constructing gropes.



\begin{prop}\label{prop:effectinfections}
Let $R(-,\psi) \colon \mathcal{C}^m_{SL} \to \mathcal{C}^{\ell}$ be a satellite operator as in Definition~\ref{def:stringlinkinfection} (where $R$ is a link or string link; in the latter case we have a function $\mathcal{C}^m_{SL} \to \mathcal{C}^{\ell}_{SL}$)
with geometric winding numbers $(w_1,\dots,w_m)$, where the link $\eta(R,\psi)$ bounds a symmetric grope $G_\eta$ of height $h$ in $(S^3 \smallsetminus (R \cup \Int(\psi(\mathbb{B}))))\times [0,1]$. Let $L_0, L_1$ be string links that are grope cobordant via a branch-symmetric grope $G_L=(G_{L_1},\dots,G_{L_m}) \subset D^2 \times I \times I$. Then $R(L_0,\psi)$ and $R(L_1,\psi)$ are grope cobordant via a branch-symmetric grope $G_{R(L)}$, that, loosely speaking, is formed from $w_j$ copies of $G_{L_j}$ for $1\leq j\leq m$, and multiple copies of the components of $G_\eta$ attached to tips of the copies of $G_L$.  More specifically,
\begin{enumerate}[(A)]
\item\label{item:effect-infections-A}  The genus of the first stage surface of $G_{R(L)_s}$ is
\begin{equation}\label{eq:genuschange}
g_1(G_{R(L)_s})=\sum_{j=1}^\ell w_j^sg_1(G_{L_j});
\end{equation}
where  $w_j^s$ is the number of strands of $R_s$ that pass through the $j^{th}$ subdisk of $\mathbb{D}$.
\item\label{item:effect-infections-B}  A branch, $B'$, of the new grope $G_{R(L)}$  consists (abstractly) of a copy of a branch, $B$, of $G_L$ along with (a boundary connected sum of) copies of $G_\eta$ attached to each tip of $B$. Thus the length of $B'$ is $h$ more than the length of $B$; and
\item\label{item:effect-infections-C} for each tip of $B$ the number of copies of $G_{\eta_j}$  used is equal to the $j^{th}$ cap multiplicity i.e.\ to the number of meridians of $L_j$ occurring in the (punctured) cap chosen for this tip of $B$.
\end{enumerate}
Moreover, in the special case that $R$ is a slice link and $L_1$ is a trivial string link, so $R(L_1,\psi)$ bounds slice disks $\Delta\hookrightarrow B^4$, then, by appending $\Delta$ to $G_{R(L)}$, we have that $R(L_0,\psi)$ bounds a grope in $B^4$; and for this case the weaker condition that the link $(\eta_1,\dots,\eta_m)$ bounds a symmetric grope of height $h$ in $B^4-\Delta$ is sufficient.
\end{prop}

The techniques used in the following proof are very similar to those used in~\cite[Prop. 3.4, Corollary 3.14]{CT} and~\cite[Thm 3.4]{Hor2}.  The difference is that we keep precise track of the genera and the number of copies used in the construction.

\begin{proof} We will assume that $R$ is a link in $S^3$.  The proof is essentially the same if $R$ is a string link.

First, we will describe a simple grope concordance $G'$ between $R(L_0)$ and $R(L_1)$. By hypothesis the string link $L_0\hookrightarrow D^2 \times I \times  \{0\}$ is grope cobordant via $G_L$, in $D^2 \times I \times I$ to the string link $L_1$ in $D^2 \times I \times \{1\}$. Note that by Lemma~\ref{lem:gropecaps} we can choose caps for $G_L$. We can assume that the tips of these caps are assumed to be copies of the meridians of the components of $L_1$, since in the proof of Lemma 5.2 the final tubing to the boundary can be done in either direction, to $L_0$ or to $L_1$; in particular we can choose all tubes so that they lead to meridians of $L_1$.  Henceforth in this proof we assume that these caps are part of $G_L$.  The choices of caps will affect the structure of the grope.

Let $(w_1,\dots,w_m)L$ denote the string link obtained by taking, for each $1\leq j\leq m$, $w_j$ parallel copies of the $j^{th}$ component of $L$ and then perhaps changing the string orientation of some of the copies as needed below.  Thus $(w_1,\dots,w_m)L_0$ is grope cobordant in $\mathbb{B}\times [0,1]$ to $(w_1,\dots,w_m)L_1$ via a grope that we will call $(w_1,\dots,w_m)G_L$.  The latter is obtained by taking parallel copies of the components of $G_L$.

For each strand of the $s^{th}$ component of $R$ that passes through $\eta_j$ we need a copy of the $j^{th}$ component of the grope concordance $G_L$. This copy of $(G_L)_j$ will (below) become part of the first stage surface for the $s^{th}$-component of the grope concordance $G'$. This observation justifies equation~(\ref{eq:genuschange}). Moreover this grope concordance is capped by parallel copies of the caps of $G_L$.

A key observation is that each of the cap tips of $(w_1,\dots,w_m)G_L$ will \textit{not} be a meridian of $(w_1,\dots,w_m)L_1$, but rather will be a ``fat meridian'' of say the $j^{th}$-component of $L_1$ (a circle that encloses all the parallel copies of $(L_1)_j$ which were taken).  So it is best to think of taking parallel copies of the components of $L_1$ that lie inside the original tubular neighborhood of $L_1$, so that these new fat meridians are actually the same as the meridians of the original components of $L_1$.

Now let $\mathcal{B}$ denote the complementary $3$-ball to $\psi(\mathbb{B})$, meaning $S^3=\mathcal{B}\cup \psi(\mathbb{B})$ where $\psi\colon \mathbb{B} \rightarrow S^3$ is an extension of $\psi \colon \mathbb{D} \rightarrow S^3$  as described in Definition~\ref{def:string_link}.
 Then, by definition of string link infection, $(S^3,R(L_k))$, for $k=0,1$, decomposes as $$(\psi(\mathbb{B}),(w_1,\dots,w_m)L_k)\cup(\mathcal{B},R\cap \mathcal{B}),$$ for a certain choice of string orientations on the components of $(w_1,\dots,w_m)L_k$. Now define a grope concordance $G'$ from $R(L_0)$ to $R(L_1)$ by
$$
G'\equiv (w_1,\dots,w_m)G_L \cup \left ((R\cap \mathcal{B}) \times [0,1]\right)\hookrightarrow \left(\mathbb{B}\times [0,1]\right)\cup \left(\mathcal{B}\times [0,1]\right)\equiv S^3\times [0,1].
$$
Recall that the boundaries of the punctured caps of $G'$ are copies of the meridians of the original components of $L_1$ in $\mathbb{B}\times \{1\}$. But these are identified, in the process of string link infection, with copies of the circles $\eta_j$ in $S^3\times \{1\}$.

Now we will add  extra stages to the grope concordance $(S^3\times [0,1], G')$.  First extend $G'$ to $(S^3\times [0,2], G'')$ by adding the product annuli $R(L_1)\times [1,2]\hookrightarrow S^3\times [1,2]$. By hypothesis $(\eta_1,\dots,\eta_m)\hookrightarrow S^3\times \{1\}$ bounds a symmetric grope $G_\eta$ of height $h$ in $(S^3 \smallsetminus (R \cup \Int(\psi(\mathbb{B}))))\times [1,2]$
and hence certainly bounds such a grope in the exterior of the product annuli $R(L_1)\times [1,2]\hookrightarrow S^3\times [1,2]$. So, finally, we can add copies of the $\pm G_{\eta_j}\hookrightarrow S^3\times [1,2]$ to each of the copies of $\pm \eta_j$ that occur as tips of the  caps of $G''$.
The resulting grope, which we call $G_{R(L)}$, is a grope concordance from $R(L_0)$ to $R(L_1)$ each of whose branches has length $h$ more than that of the corresponding branch of $G_L$.
A schematic of the construction is shown in Figure~\ref{figure:doubling-op-proof-2}.
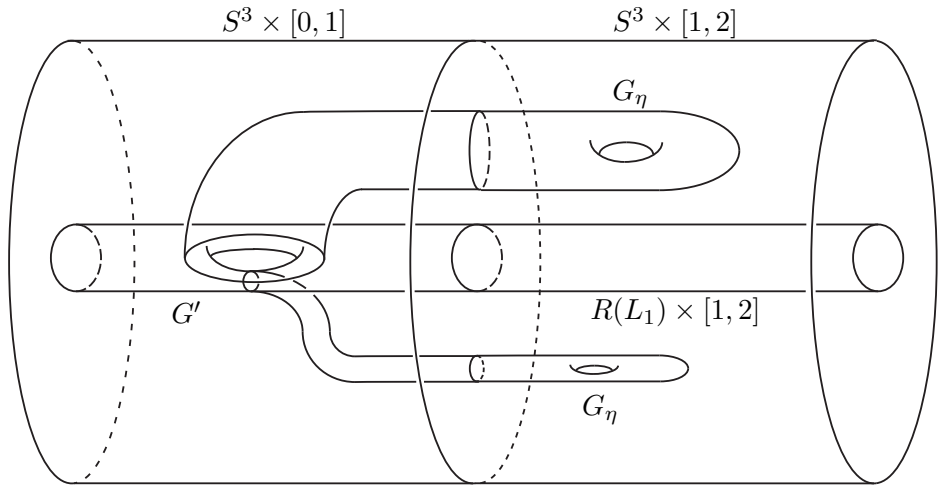
\begin{figure}[h]
\begin{center}
\begin{tikzpicture}
\node[anchor=south west,inner sep=0] at (0,0){\includegraphics[scale=0.35]{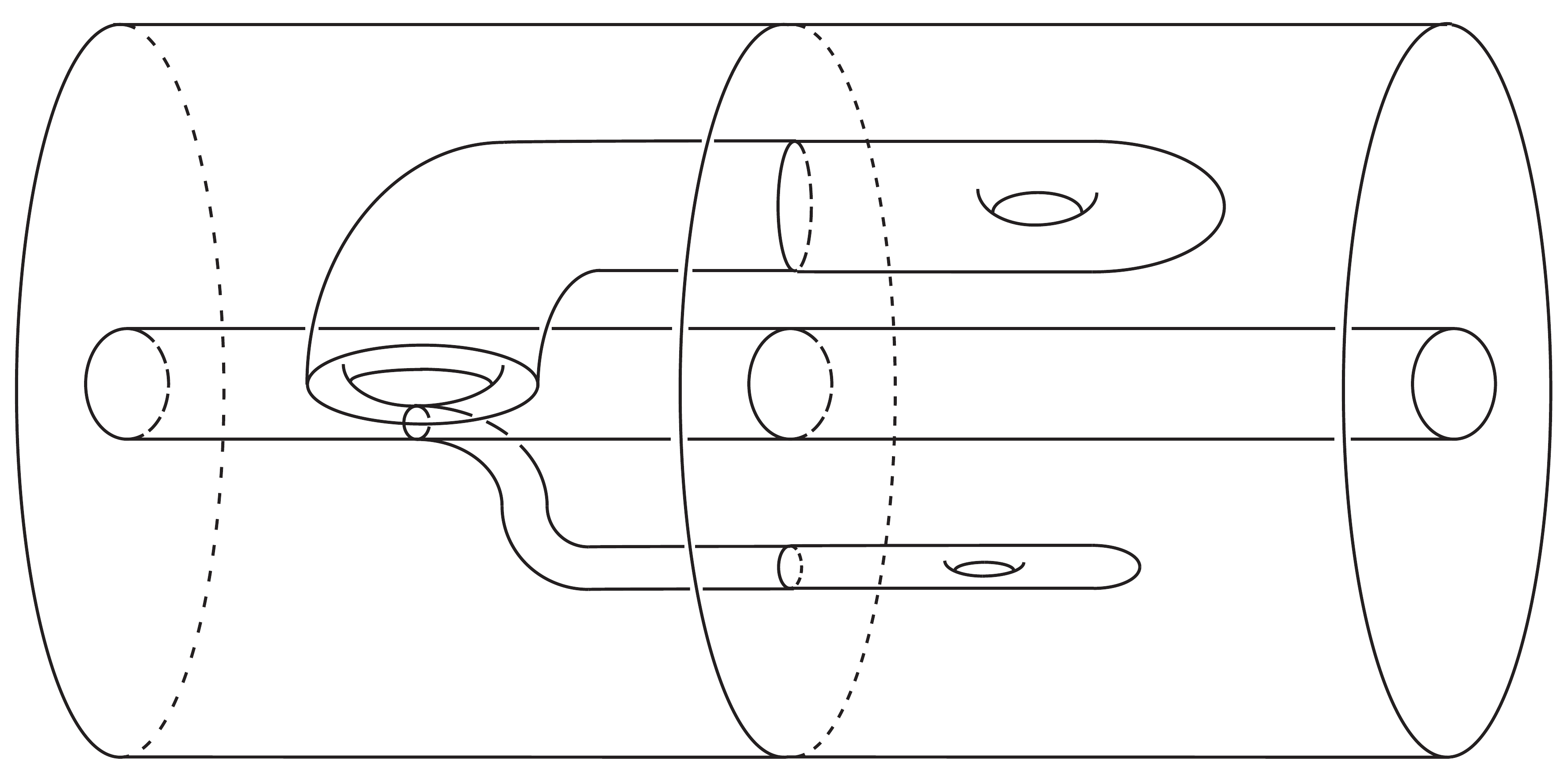}};
\node at (2.5,2.5)  {$G'$};
\node at (8.4,5.4)  {$G_\eta$};
\node at (8,1.2)  {$G_\eta$};
\node at (9,6.35)  {$S^3 \times [1,2]$};
\node at (3.8,6.35)  {$S^3 \times [0,1]$};
\node at (9,2.5)  {$R(L_1) \times[1,2]$};
\end{tikzpicture}
\end{center}
     \caption{A schematic of the construction of $G_{R(L)}$.  The links $R(L_i)$ are drawn as single circles and each component grope is drawn as a genus one surface.}
    \label{figure:doubling-op-proof-2}
\end{figure}

Moreover, in the case that $R$ is a slice link and $L_1$ is a trivial string link, so $R(L_1)$ is a slice link which bounds some slice disks $\Delta\hookrightarrow B^4$, then, by appending $\Delta$ to $G_{R(L)}$ constructed above, we have that $R(L_0)$ bounds a grope in $B^4$ with essentially the same topology and combinatorics as above.  Moreover, for this case the weaker condition that the link $(\eta_1,\dots,\eta_m)$ bounds a symmetric grope of height $h$ in $B^4-\Delta$ is sufficient, as can be seen by analyzing the previous paragraph.
\end{proof}

The following corollary generalizes~\cite[Prop. 3.4, Corollary 3.14]{CT}, \cite[Prop. 4.7]{Cha2014} and \cite[Thm 3.4]{Hor2}.
Recall that we denote the set of $m$ component links that bound a grope of height $n$ in $D^4$ by $\mathcal{G}_n^m$, and we say that a string link $J \in \mathcal{G}_n^m$ if $\widehat{J} \in \mathcal{G}_n^m$.

\begin{corollary}\label{cor:gropelengthincrease} Let $R=(R_1,\dots,R_\ell)$ be a slice link (or a slice string link) that admits a system of slice disks $\Delta$. Suppose the link $(\eta_1,\dots,\eta_m)$ is the data of a string link infection as above and bounds a symmetric grope of height $h$ in the exterior of $\Delta$. Suppose $L\in \mathcal{G}_n^m$. Then $R(L)\in \mathcal{G}_{n+h}^\ell$.
\end{corollary}

\begin{proof} The hypotheses of Corollary~\ref{cor:gropelengthincrease} are the hypotheses of the special case in the last sentence of Proposition~\ref{prop:effectinfections}.
\end{proof}

Recall that a \emph{doubling operator} is a special case of a string link infection wherein $R$ is a slice link (or a slice string link)  and $\ell k(\eta_j,R_k)=0$ for all $j,k$; see for example~\cite[p.1598]{CHL4},\cite[p.1425]{CHL3} and \cite[Def.1.3]{Burke2014}.  The latter condition says that the algebraic winding number is zero.

Consider the special case of a string link infection when $R=P$ is a knot and $m=1$, so $\eta = \eta_1$ is a single curve.  We can think of $P$ as a knot in the solid torus $\ol{S^3 - (\mathbb{B} - (D_1 \times I))}$, called the \emph{pattern knot}.  Then $P(-,\eta) \colon \mathcal{C} \to \mathcal{C}$ is called a \emph{satellite operator}.  The algebraic and geometric winding numbers are now both just single numbers.

\begin{corollary}\label{cor:doublingincreaselength} For any doubling operator $(R,\eta)$, $R(\mathcal{G}_n^m)\subset \mathcal{G}_{n+1}^\ell$. In particular for any winding number zero pattern knot $P$ which is a slice knot when viewed in $S^3$, the induced (faithful) satellite operator $P\colon \C\to \C$ satisfies $P(\mathcal{G}_n)\subset \mathcal{G}_{n+1}$ for each $n$.
\end{corollary}

Here faithful refers to the lack of potential twistings that could occur during a satellite operation; see \cite[p.~111]{R}.

\begin{proof} Since $(R,\eta)$ is a doubling operator, $R=(R_1,\dots,R_\ell)$ is an $\ell$-component slice link.  Let $\Delta$ be a collection of slice disks. Moreover  $(\eta_1,\dots,\eta_m)$ forms a trivial link in $S^3-R$, for which $\ell k(\eta_j,R_k)=0$ for all $j,k$. Hence the following lemma can be applied to find a collection of disjointly embedded height one gropes with boundary the $\eta_i$. The proof is then finished by applying Corollary~\ref{cor:gropelengthincrease} with $h=1$.
\end{proof}

\begin{lemma}\label{lem:doublingbounds}
For any  link (or string link) $R$, and any set of disjoint circles $\{\eta_1,\dots,\eta_m\}$ in the exterior of $R$ for which $\ell k(\eta_i,\eta_j)=\ell k(\eta_i,R_k)=0$ for all $i,j,k$, the link $\{\eta_1,\dots,\eta_m\}$ bounds a height one grope whose interior lies in $(S^3-R-\eta_1-\eta_2-\dots-\eta_m)\times [0,1]$. Moreover, for any set of slice disks $\Delta$ for a link (or string link) $R$, the link $\{\eta_1,\dots,\eta_m\}$ bounds a height one grope in $B^4-\Delta$.
\end{lemma}

\begin{proof} For each $\eta_j$, choose a Seifert surface $S_j$ whose interior is disjoint from the other $\eta_i$.  For each $j$, let $F_j\hookrightarrow S^3\times [0,j]$ be the surface bounding $\eta_j$ which consists of the product annulus $\eta_j\times [0,j]$ together with a copy of $S_j\hookrightarrow S^3\times \{j\}$. Since the $S_j$ occur in different levels, these surfaces  in $S^3\times [0,m]$ are disjoint. Moreover, if $i\ne j$ then $F_j$ is disjoint from $\eta_i\times [0,m]$. After a slight adjustment along the annulus part of $F_j$, we can assume it is disjoint from $\eta_j\times [0,m]$ except where they coincide at $\eta_j\times \{0\}$.  After smoothing corners, we may assume that the $F_j$ are transverse to each component of $R\times [0,m]$, hence intersect each component  in pairs of points with opposite signs. Using disjoint arcs in these components as guides, we can alter each $F_j$ by adding tubes to get it disjoint from $R\times [0,m]$. This yields the desired height one grope.
A schematic of the proof is shown in Figure~\ref{figure:doubling-op-proof-1}.
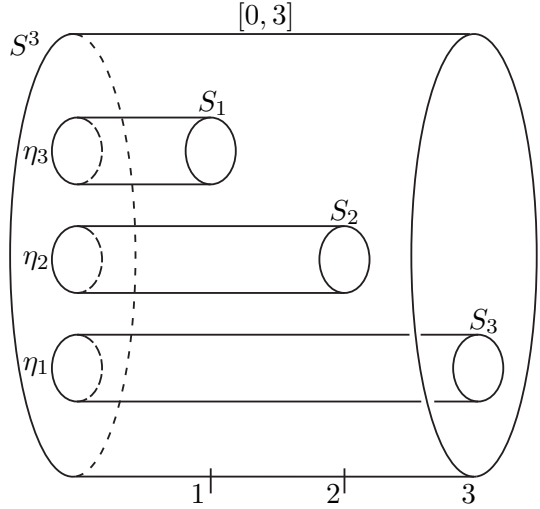
\begin{figure}[h]
\begin{center}
\begin{tikzpicture}
\node[anchor=south west,inner sep=0] at (0,0){\includegraphics[scale=0.35]{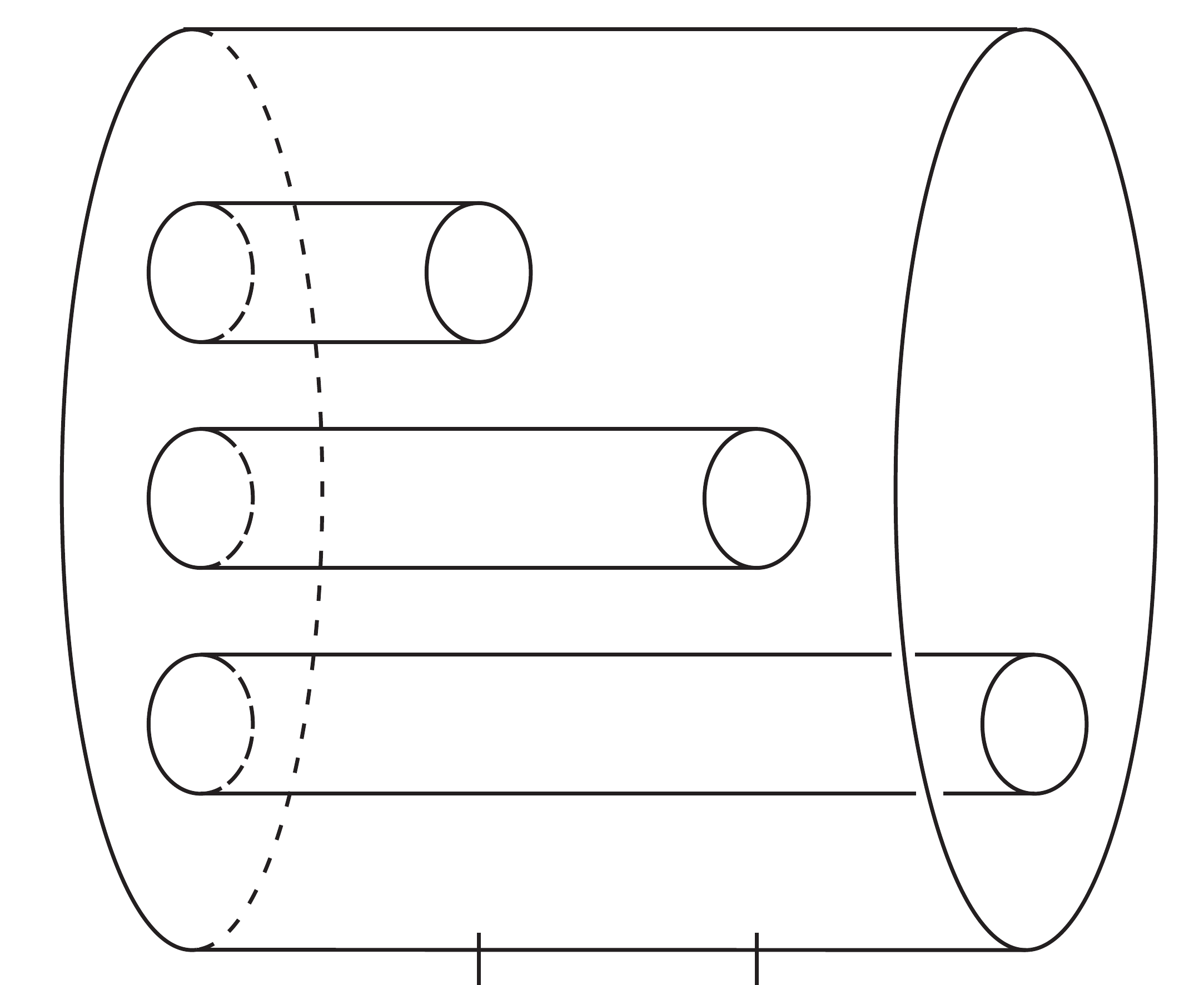}};
\node at (2.9,0)  {$1$};
\node at (4.7,0)  {$2$};
\node at (6.5,0)  {$3$};
\node at (0.75,1.7)  {$\eta_1$};
\node at (0.75,3.1)  {$\eta_2$};
\node at (0.75,4.5)  {$\eta_3$};
\node at (0.6,6)  {$S^3$};
\node at (3.8,6.35)  {$[0,3]$};
\node at (3.1,5.2)  {$S_1$};
\node at (4.85,3.75)  {$S_2$};
\node at (6.7,2.3)  {$S_3$};
\end{tikzpicture}
\end{center}
     \caption{A schematic of the proof for $m=3$.}
    \label{figure:doubling-op-proof-1}
\end{figure}

The case of slice disks is easier and actually follows from the above case by assuming that the slice disks are products $R\times [0,\epsilon]$  near the boundary, and finding the grope in  $R\times [0,\epsilon]$ as above.
\end{proof}

\begin{definition}\label{def:Lipshitz} A map $f\colon(X,d)\to (Y,d')$ between pseudo-metric spaces is  \emph{Lipschitz continuous} if there exists some $\delta\geq 0$, called a \emph{Lipschitz constant} of $f$,  such that $d'(f(x),f(w))\leq \delta d(x,w)$ for all $x,w\in X$.
\end{definition}

\begin{prop}\label{prop:Lipschitz} For any pattern knot $P$  and any $q$, the satellite operator $P \colon (\C,d^q)\to(\C,d^q)$ is Lipschitz continuous with Lipschitz constant equal to the geometric winding number of $P$.
\end{prop}

\begin{proof} Let $\gw(P)$ be the geometric winding number of $P$. Suppose that $K$ and $J$ are arbitrary knots and $\Sigma$ is any grope concordance between them. From $\Sigma$ we will construct a grope concordance $G$ between $P(K)$ and $P(J)$. It will suffice to show that $\|G\|^q\leq \gw(P)\|\Sigma\|^q$ since this implies that $d^q(P(K),P(J))\leq \gw(P) d^q(K,J)$,  which implies that $P$ is Lipschitz.

To construct $G$, we essentially repeat the first part of the proof of Proposition~\ref{prop:effectinfections}. In this simple case, a satellite operation is a string link infection where $\mathbb{B}$ is simply a thickening of the meridional disk of the solid torus in which $P$ lies. By hypothesis the knotted arc $K\hookrightarrow \mathbb{B}\times  \{0\}$ is grope cobordant via $\Sigma$, in $\mathbb{B}\times [0,1]$ to the knotted arc $J$ in $\mathbb{B}\times \{1\}$.  Taking $\gw(P)$ parallel copies of $\Sigma$ and changing some orientations, we see that the string link $\gw(P)K$ is grope cobordant in $\mathbb{B}\times [0,1]$ to $\gw(P)J$ via a grope that we will call $\gw(P)\Sigma$.   Let $\mathcal{B}$ denote the complementary $3$-ball to $\mathbb{D}$. Now define a grope concordance $G$ from $P(K)$ to $P(J)$ by
$$
G\equiv \gw(P)\Sigma \cup \left ((P\cap \mathcal{B}) \times [0,1]\right)\hookrightarrow \left(\mathbb{B}\times [0,1]\right)\cup \left(\mathcal{B}\times [0,1]\right)\equiv S^3\times [0,1].
$$
 The grope $G$ is composed of $\gw(P)$ copies of $\Sigma$ (some with altered orientation) banded together along their boundaries. Thus the genus of the first stage surface of $G$ is $\gw(P)$ times that of $\Sigma$ and each branch of $\Sigma$ is repeated $\gw(P)$ times.
Thus, for any $q$,
$$
\|G\|^q=~\gw(P)~\|\Sigma\|^q,
$$
as desired.
\end{proof}

\noindent A special case of Lipschitz continuity is the following.

\begin{definition}\label{def:contraction}
A map $f\colon (X,d)\to (Y,d')$ between pseudo-metric spaces is called a \emph{contraction mapping} if there exists some $0\leq \delta<1$ such that $d'(f(x),f(w))\leq \delta d(x,w)$ for all $x,w\in X$.
\end{definition}

\begin{prop}\label{prop:contraction} For any winding number zero satellite operator $R(-,\eta)$ there is an $N$, depending only on the geometric winding numbers of $R(-,\eta)$, such that for each $q>N$, $R\colon (\C^m_{SL},d^q)\to(\C^\ell,d^q)$ is a contraction mapping. In particular, for any winding number zero pattern knot $P$, and any $q>\gw(P)$, the satellite operator $P\colon (\C,d^q)\to(\C,d^q)$ is a contraction mapping.
\end{prop}

This gives further evidence that $\mathcal{C}$ has the structure of a fractal space as conjectured in \cite{CHL5}.

\begin{proof} Suppose that $R=(R_1,\dots,R_\ell)$ and $\eta=(\eta_1,\dots,\eta_m)$. Let $w_j^s$, for $1\leq s\leq \ell$ and $1\leq j\leq m$, be the number of strands of $R_s$ that ``pass through'' $\eta_j$. Let $N=\max_j\{\sum_{s=1}^\ell w_j^s\} = \max_j \{w_j\}$. Fix any $q>N$. We will show that
$R\colon(\C^m,d^q)\to(\C^\ell,d^q)$ is a contraction mapping; here we omit $\eta$ from the notation of the operator for brevity.

In the case that $N=0$ then all $w_j^s=0$ so, for any $L$, $R(L)=R$. Thus $d^q(R(L_0),R(L_1))=0$ for all $L_0, L_1$, so $R$ is a contraction mapping  for $\delta=0$.

Henceforth we assume that $N>0$.  Let $\delta=\frac{N}{q}$, so $0\leq \delta<1$. We will show that, for all string links $L_0, L_1$,
\begin{equation}\label{eq:contract}
d^q(R(L_0),R(L_1))\leq \delta d^q(L_0,L_1).
\end{equation}
Since $R(-,\eta)$ has winding number zero,  Lemma~\ref{lem:doublingbounds} ensures that the link $(\eta_1,\dots,\eta_m)$ bounds a symmetric grope $G_\eta$ of height $1$ in  $(S^3-(R \cup \mathbb{B}))\times [0,1]$.  Suppose that $L_0$ is grope cobordant to $L_1$ via a grope $\Sigma$.  Recall that
$$\|\Sigma\|^q=\sum_{j=1}^m\sum_{i=1}^{g_1(\Sigma_j)}\frac{1}{q^{n_{ji}}} \bigg( 1- \sum_{k=2}^{n_{ji}+1} \frac{1}{g_k^{ji}}\bigg)=\sum_{j=1}^m\|\Sigma_j\|^q$$
Let $G$ denote the grope concordance from $R(L_0)$ and $R(L_1)$ constructed in Proposition~\ref{prop:effectinfections}. Hence
$$
g_1(G_s)=\sum_{j=1}^mw_j^sg_1(\Sigma_j);  ~1\leq s\leq \ell.
$$
Moreover $G_s$ is formed from $w_j^s$ parallel copies of $\Sigma_j$ with copies of the surfaces $G_\eta$ attached to all tips (and sum over $j$). The key point is that, for fixed  $j$ the topology of all of these copies is the same, independent of $s$. Thus
\begin{align*}
d^q(R(L_0),R(L_1))\leq & \|G\|^q= \sum_{s=1}^\ell\|G_s\|^q= \sum_{s=1}^\ell\sum_{j=1}^mw_j^s\|\Sigma_j\cup \text{copies } G_\eta\|^q= \\
&\sum_{s=1}^\ell\sum_{j=1}^mw_j^s\sum_{i=1}^{g_1(\Sigma_j)}\frac{1}{q^{n_{ji}+1}} \bigg( 1- \sum_{k=2}^{n_{ji}+1} \frac{1}{g_k^{ji}}-\frac{1}{g_{n_{ji}+2}^{ji}}\bigg).
\end{align*}
Now if we just ignore the last terms arising from the final stage surfaces, we can continue with:
\begin{align*}
\leq & \sum_{s=1}^\ell\sum_{j=1}^mw_j^s\sum_{i=1}^{g_1(\Sigma_j)}\frac{1}{q^{n_{ji}+1}} \bigg( 1- \sum_{k=2}^{n_{ji}+1} \frac{1}{g_k^{ji}}\bigg) \\
= & \sum_{s=1}^\ell\sum_{j=1}^m \frac{w_j^s}{q}\|\Sigma_j\|^q=\frac{1}{q} \sum_{j=1}^m \|\Sigma_j\|^q\bigg(\sum_{s=1}^\ell w_j^s\bigg) \\
\leq & \frac{N}{q}\|\Sigma\|^q=\delta\|\Sigma\|^q.
\end{align*}

Hence $d^q(R(L_0),R(L_1))\leq \delta\|\Sigma\|^q.$
Since this is true for \textit{any} grope concordance $\Sigma$ from $L_0$ to $L_1$, this establishes inequality~(\ref{eq:contract}).

In particular, any pattern knot $P$ with winding number zero  is an example of such a string link operator with $\ell=m=1$ and $N=\gw(P)$.
\end{proof}

\section{Examples exhibiting the non-discrete behaviour}\label{section:examples-for-non-discrete}

\begin{proposition}\label{ex:specificdecreasing}
For each $m\geq 3$, there exists a family of knots, $\{K_n^m~|~n\geq 0\}$, such that $K_n^m$ bounds a symmetric height $n+2$ grope, whose first stage has genus~$2^{n}$ and whose higher stages have genus one. Moreover $K_n\notin \mathcal{G}_{n+3}$.
\end{proposition}

Note that Proposition~\ref{ex:specificdecreasing} completes the proof of Theorem~\ref{cor:notdiscrete}.

\begin{proof}
 First, for any fixed $m$, we will recursively define knots $K_n^m$, each of which bounds a symmetric  grope of height $n+2$.  Let $K_0=K_0^m$ be  the mirror image of the knot (independent of $m$) given in ~\cite[Figure 3.6]{CT}, that bounds a symmetric height $2$ grope whose first and second stage surfaces each have genus one. The other relevant property of $K_0$ is that $\rho_0(K_0)$, the integral of its Levine-Tristram signature function, is negative, as shown in  ~\cite[Lemma 4.5]{CT}. Let $R^m$ be the ribbon knot shown on the left-hand side of Figure~\ref{fig:examplesKn}. The $-m$ in the box indicates $m$ full left-handed twists between the bands.  Below the dotted $\eta$ circle on the left of Figure~\ref{fig:examplesKn} one sees what we mean by half a negative twist.  Let $\eta_m$ be an oriented unknotted circle which has linking number zero with $R^m$. An example of a curve with this property is shown as dashed in the figure. The actual $\eta_m$ we will use will be described presently. Then let $K_n^m\equiv R^m(K_{n-1}^m)$ as shown on the right-hand side of Figure~\ref{fig:examplesKn}, for the particular choice of $\eta$ in the left-hand diagram.
 \begin{figure}[htbp]
\setlength{\unitlength}{1pt}
\begin{picture}(327,151)
\put(-15,10){\includegraphics{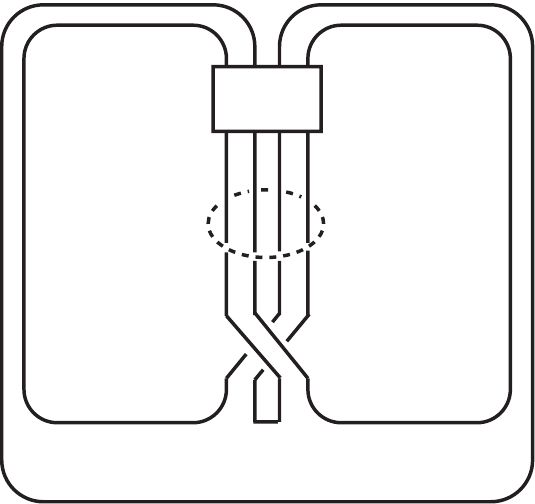}}
\put(198,10){\includegraphics{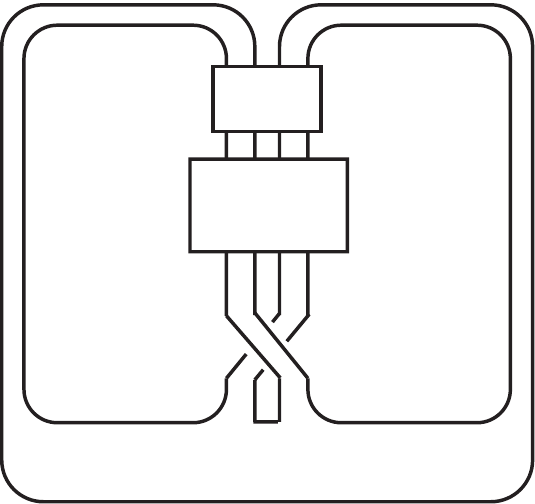}}
\put(51,123){$-m$}
\put(87,84){$\eta$}
\put(262,92){$K_{n-1}^m$}
\put(264,123){$-m$}
\put(233,-5){$K_n^m\equiv R^m(K_{n-1}^m)$}
\put(50,-5){$(R^m,\eta)$}
\end{picture}
\caption{}\label{fig:examplesKn}
\end{figure}
Henceforth we will suppress $m$ from the notation until it becomes relevant. For any such $\eta$, the pair $(R,\eta)$ implicitly symbolizes a winding number zero pattern, since the exterior of a neighborhood of $\eta$ is a solid torus and $R$ gives a pattern knot inside this solid torus, which, by abuse of notation, we will also call $R$. Thus $R(-,\eta)$ defines a doubling operator. Since $K_0\in \mathcal{G}_2$, by repeatedly applying Corollary~\ref{cor:doublingincreaselength}, we see that $K_n\in \mathcal{G}_{n+2}$. By the second sentence of Proposition~\ref{prop:contraction}, for any $q>\gw(R)$,~$R$ is a contraction operator, so the sequence $K_n$ converges to the class of the trivial knot in $(\C,d^q)$. Below we will show that, for certain choices of $\eta$, $K_n\notin \mathcal{G}_{n+3}$ which implies that each knot in the sequence represents a distinct concordance class. This will be enough to show that $(\mathcal{C},d^q)$ does not have the discrete topology for any $q>\gw(R)$.

However, to prove Corollary~\ref{cor:notdiscrete} for all $q>1$ we must choose $\eta_m$ very carefully and prove that $K_n$ bounds a symmetric height $n+2$ grope whose first stage has genus $2^{n}$ and whose higher stages have genus one. For this we must  look more carefully at the gropes constructed in the proof of Proposition~\ref{prop:effectinfections}. In particular we are in the special case covered in the last paragraph of the statement of that proposition.  The existence of the $\eta_m$ we require is guaranteed by the following lemma.

\begin{lemma}\label{lem:goodeta} Let $R$ be a knot with cyclic rational Alexander module. Then there exist a curve $\eta$ with the following properties:
\begin{enumerate}[(1)]
\item  $\eta$ bounds an embedded disk in $S^3$ that intersects $R$ transversely in two points with opposite signs;
\item  $\eta$ generates the rational Alexander module of $R$;
\item $\eta$ bounds, in the exterior of $R$, an embedded genus one surface with symplectic basis $x,y$, each of which bounds a cap in $S^3$ that intersects $R$ precisely once.
\end{enumerate}
\end{lemma}
\begin{proof}[Proof of Lemma~\ref{lem:goodeta}] The rational Alexander module of $R^m$, henceforth denoted by $\mathcal{A}$, is cyclic with order $(mt-(m+1)((m+1)t-m)$. Let  $\alpha$ be a generator. Since $\Q[t,t^{-1}]$ is a PID and since $t^1+t^{-1}-2$ is relatively prime to the order of $\mathcal{A}$ ($t^{\pm 1} -1$ cannot be a factor of any Alexander polynomial of a knot), there is a class $\beta\in \mathcal{A}$ such that
\begin{equation}\label{eq:eta1}
\alpha=(t_*^1-1+t_*^{-1}-1)\beta.
\end{equation}
After possibly by multiplying $\alpha$ and $\beta$  by the same positive integer, we may assume without loss that $\alpha$ and $\beta$ come from the integral Alexander module and hence are represented by homotopy classes $\tilde{\eta}$ and $b$ in $\pi_1(S^3-R^m)^{(1)}$, and that equation~\eqref{eq:eta1} holds in the integral Alexander module. The desired circle $\eta$ will be a particular representative of the homotopy class $\tilde{\eta}$. By ~\eqref{eq:eta1} we have the following statement in $\pi_1(S^3-R)$:
$$
\tilde{\eta}\gamma=(\mu b\mu^{-1})b^{-1}(\mu^{-1}b\mu)b^{-1}=[\mu, b\mu^{-1}b^{-1}],
$$
where $\mu$ is the homotopy class of a fixed meridian of $R^m$ and $\gamma\in \pi_1(S^3-R)^{(2)}$. Redefine $\tilde{\eta}=\tilde{\eta}\gamma$ since they both represent $\alpha$. Now the proof follows that of~\cite[Lemma 3.9]{CT}. Of course $\mu$ is represented by an actual geometric meridian, by which we mean  an oriented circle in $S^3- R$ that bounds an embedded disk $D_1$, that hits $R$ in only one point and which contains the basepoint in its boundary. We can represent $b\mu^{-1}b^{-1}$ by another such meridian, that is there is an embedded disk, $D_2$, whose boundary represents this class, which intersects $R$ in one point, and which intersects $D_1$ only at the basepoint. These two disks will be the caps of the punctured torus we now define. The circles $x\vee y\equiv\partial D_1\cup \partial D_2$ form an embedded wedge of circles. We can thicken each circle to get two plumbed annuli, forming an embedded punctured torus, $G_\eta$, whose boundary realizes the commutator $\tilde{\eta}$ (see ~\cite[Lemma 3.9]{CT} for a figure and more details). Let $\eta$ be the boundary of $G_\eta$. Then $\eta$ satisfies $2.$ and $3$. Moreover the torus can be surgered along either $D_1$ or $D_2$ showing that $\eta$ also bounds an embedded disk that hits $R$ in two points.  This completes the proof of Lemma~\ref{lem:goodeta}.
\end{proof}

Now assume that $\eta$ satisfies the properties in Lemma~\ref{lem:goodeta}.  Thus $\eta$ bounds a genus one height one grope $G_\eta$ in the exterior of a slice disk $\Delta$ for $R$. Moreover this grope is capped by disks $D_1$, $D_2$, each of which intersects $\Delta$ in single point. Now suppose, inductively, that $K_{n-1}$  bounds a symmetric height $(n+1)$ grope $G$ whose first stage has genus $2^{n-1}$,  whose higher stage surfaces all have genus one and moreover where each tip of $G$ has a cap with multiplicity $1$.  Thus the punctured caps are merely ``pushing'' annuli. We will apply the last paragraph of Proposition~\ref{prop:effectinfections} with $h=1$.   By part (\ref{item:effect-infections-A}) of Proposition~\ref{prop:effectinfections}, $R(K_{n-1})$  bounds a  grope $\Sigma$ whose first stage has genus twice that of  $G$, namely $2^n$. Moreover, by part (\ref{item:effect-infections-B}) of Proposition~\ref{prop:effectinfections},  the union of all except the top stage surfaces of $\Sigma$ consists of a boundary connected sum of two parallel copies of $G$, so all of the surfaces in stages $2$ through $n+1$ are genus one surfaces.   By part (\ref{item:effect-infections-C}) of Proposition~\ref{prop:effectinfections}, the $(n+2)^{th}$-stage of $\Sigma$ consists of these pushing annuli together with one copy of the punctured torus $G_\eta$ per annulus. Thus  $K_{n}$  bounds a symmetric height $n+2$ grope $\Sigma$  whose first stage has genus $2^{n}$ and  whose higher stage surfaces all have genus one. Moreover each tip of $\Sigma$ is capped by a copy of $D_1$ or $D_2$ which have multiplicity one. This finishes the inductive step of the proof that $K_{n}$  bounds a symmetric height $n+2$ grope $\Sigma$  whose first stage has genus $2^{n}$ and  whose higher stage surfaces all have genus one. The base case, namely that $K_0$  bounds a symmetric height $2$ grope whose first stage has genus $2^{0}$,  whose second stage surfaces  have genus one and where each tip  has a cap with multiplicity $1$, was shown in ~\cite[Figure 3.13]{CT}.

It only remains to show that, if $m\geq 3$, $K_n^m\notin \mathcal{G}_{n+3}$.  First, we claim that each $(R^m,\eta_m)$ is a \textit{robust doubling operator} in the sense of ~\cite[Def. 7.2]{CHL5}.
 This requires that $\mathcal{A}(R^m)$ is cyclic with order $p(t)p(t^{-1})$ where $p(t)$ is prime, generated by $\eta$.  We have these properties. Moreover we must verify that,  for each isotropic submodule $P\subset \mathcal{A}(R^m)$, either $P$ is a Lagrangian arising from the kernel of the inclusion to a ribbon disk exterior, or else the corresponding \textit{first-order $L^{(2)}$-signature} of $R^m$, $\rho(M_R, \phi_P)$  is non-zero ~\cite[Def. 7.1]{CHL5}.  In our case $R^m$ has two Lagrangian subspaces corresponding to two ribbon disks and only one other isotropic submodule, namely $P=0$. In this case the corresponding first-order signature is denoted $\rho^1(R^m)$. The details to justify these statements are in ~\cite[Example 7.3]{CHL5}. The fact that if $m\geq 3$ then $\rho^1(R^m)<0$ is shown in  ~\cite[Theorem~7.2.1]{CDavisThesis}.   Thus $K_n^m$ is the result of applying $n$ successive robust doubling operators to the Arf invariant zero knot $K_0$.   Hence  ~\cite[Theorem 7.5]{CHL5} (or more precisely its proof) can be applied to $K_n^m$ to deduce that no positive multiple $cK^m_n$ lies in $\mathcal{F}_{(n.5)}$ as long as the equation $d\rho_0(K_0)+\rho(R^m,\phi_P)=0$ has no solution for $1\leq d\leq c$  and for any first-order signature of $R^m$ (see equation $(7.8)$ of the proof of ~\cite[Theorem 7.5]{CHL5}). Since we saw that $\rho_0(K_0)<0$, and since the set of first order signatures of $R^m$  is $\{0,0,\rho^1(R^m)\}$, this equation cannot be solved. It follows that no nonzero multiple of $K^m_n$ lies in $\mathcal{F}_{(n+1)}$ and hence no non-zero multiple of $K^m_n$ lies in $\mathcal{G}_{(n+3)}$  by ~\cite[Theorem 8.11]{COT}.
\end{proof}

\begin{remark}
  Since  the Alexander polynomials of the $R^m$ are (strongly) coprime, the set  $\{K^m_n~|~n\geq 0, m\geq 3\}$ is linearly independent in $\C$ by ~\cite[Theorem 7.7]{CHL5} (with the proof modified similarly to above).
\end{remark}

\begin{scholium}
For any $q>1$ there exist uncountably many sequences of knots $\{K_n^{m_n}\}_{n=1}^{\infty}$ such that $\|K_n^{m_n}\|^q > 0$ for all $n$ and $m_n$ but whose norms satisfy $$\lim_{n \to \infty}\|K_n^{m_n}\|^q = 0.$$  Every pair of elements in all of these sequences represent different concordance classes.
 \end{scholium}

\begin{proof}
  Alter the numbers $m_n$ used in the infinite sequence of knots whose norms tend to zero.  There are at least as many different choices of sequences of $m_n$ as there are real numbers.
\end{proof}

\section{The link concordance space is non-discrete for $q=1$}\label{section:links-q-equals-one}

In the case of knots, in order to show that the spaces of topological and smooth concordance classes of knots are not discrete, we had to restrict to $q>1$.  When $q=1$, we do not know that $d^1(K,U)=0$ for all topologically slice knots and links so the distance function $d^1$ is only well defined on $\mathcal{C}^m$, the set of smooth concordance classes of $m$-component links.  In this situation, which is somewhat complementary to that of Theorem~\ref{cor:notdiscrete} (topological, knots, $q>1$ versus smooth, links, $q=1$), we can find a sequence of links whose norms tend to zero without ever reaching zero.

\begin{proposition}
  For any $m \geq 2$ there exists a sequence of $m$-component links $L^1,L^2,\dots$ such that each link has $\|L^j\|^1 \neq 0$ but $\lim_{j \to \infty} \|L^j\|^1 = 0$.
In particular, the link concordance space $(\mathcal{C}^m,d^1)$ is not discrete.
\end{proposition}

Note that the same construction as in the following proof will also provide the proof for $q>1$, but we focus on $q=1$ since that is the outstanding case.
For $q>1$, the only difference is to replace $1/2^{n-1}$ with $1/(2q)^{n-1}$ in the proof below.

\begin{proof}
Start with a Hopf link $H=L_0 \cup L_1$, and Bing double the $L_1$ component $n$ times, to obtain a link $L = L_0 \cup BD_n(L_1)$ with $2^{n}+1$ components.  According to Lin~\cite{Lin1991} (see also \cite[Corollary~6.6]{Otto-2012}), if a link bounds a grope of height $n$, then the Milnor invariants $\ol{\mu}_L(I)$ vanish for $|I| \leq 2^n$.

The component $L_0$ bounds a symmetric grope $\Sigma_0$ of height $n$, where all the surface stages are of genus one.  In fact this grope is embedded in $S^3 \sm L$.  The norm $\|\Sigma_0\|^1 = 1/2^{n-1}$.  With $L_0$ deleted, $BD_n(L_1)$ is an unlink, so bounds a collection of disks in the interior of $B^4$, which we call unlinking disks. Denote the union of $\Sigma_0$ with the unlinking disks by $\Sigma$.  We have $\|\Sigma\|^1=1/2^{n-1}$ and thus $\|L\|^1 \leq 1/2^{n-1}$. (Strictly speaking, since we do not allow disks, we replace a small neighborhood of each of the disks with an arbitrary height grope, to obtain a sequence of gropes with length tending to zero; the resulting infimum is zero.)

Next, nonrepeating Milnor's invariants of length $2^n + 1$ are nonvanishing for $L$ by \cite[Theorem~8.1]{C4}, thus $L$ does not bound any grope of height $n+1$.  By Proposition~\ref{prop:boundedbelow} we have that $\|L\|^1 \geq 1/2^{n-1}$, so that we in fact determine $\|L\|^1 = 1/2^{n-1}$ precisely.

We need to improve this to find a sequence of links realizing a subsequence of the same sequence of norms, but for which each link in the sequence has the same number of components~$m$, for $m$ at least~$2$.
To achieve this, choose a collection of bands for the components in the complement of the unlinking disks for $BD_n(L_1)$, that connect components of $L$ as desired.  As long as the resulting multi-index, obtained from identifying the indices of banded together components, has a nonvanishing Milnor's invariant associated to it, the same lower bound for the norm applies.  For each multi-index $I$ and $m\geq 2$,  in \cite[Theorem~7.2]{C4}, the first author defined the integer $\delta(I)$ to be the least nonzero element of $\{|\ol{\mu}_L(I)|\}$, where $L$ ranges over all $m$-component links, if there is a link $L$ with $\mu_L(I)\neq 0$. Otherwise $\delta(I)=0$.  In the proof of \cite[Theorem~7.2]{C4}, for each $m\geq 2$ and each multi-index $I$, the first author showed that the Milnor's invariants of the examples we constructed above (by iterated Bing doubles and banding the components together) realize $\delta(I)$.  Moreover, for $m \geq 2$ and for any $N>0$, there exists a multi-index $I$ with $m$ distinct indices, with $|I| \geq N$, and with $\delta(I) \neq 0$ by ~\cite{O1}.  This gives rise to a sequence of $m$-component links $\{J_1,J_2,\dots\}$ with lower bounds on their norms given by $\|J_i\| \geq 1/2^{F(i)}$, for some strictly increasing function  $F\colon \mathbb{N} \to \mathbb{N}$.  That is, $J_i$ is obtained from the Hopf link by Bing doubling one of the components $F(i)+1$ times and then performing some band moves.

We claim that the function $F$ also determines an upper bound $\|J_i\|^1 \leq 1/2^{F(i)}$.
To see this, we need to see that the banded together link $J_i$ still bounds a grope of height $F(i)+1$. We construct this grope as we move into $D^4$ in the radial direction.  First perform saddle moves to cut the bands.  This yields the original $(2^{F(i)+1} +1)$-component link $L$ constructed above for $n=F(i)+1$.  Next attach the grope $\Sigma_0$ in a radial slice, then move slightly further into the 4-ball to attach the unlinking disks for $L$ with $L_0$ deleted.  This constructs a grope with the same combinatorics as before, except there are fewer link components and therefore fewer disk (or more accurately, arbitrary height grope) components to the grope.  This completes the proof of the claim.
Thus we determine the norms of the sequence of links $J_1,J_2,\dots$ precisely, as $\|J_i\|^1 = 1/2^{F(i)}$.  Since $F$ is strictly increasing, this completes the proof that the link concordance space is non-discrete.

\end{proof}

\begin{remark}
We note that, as in the knot case above, since one can choose different Milnor's invariants of length $n$ to be realised by the links constructed, there are in fact uncountably many distinct sequences of (concordance classes of) links having the property that their norms limit to zero but are all nonzero.
\end{remark}

\section{Lower bounds from higher order $\rho$-invariants}\label{section:higher-order-rho-invariants}

In this section will define lower bounds on the grope norms which can arise from all possible gropes of height $n$.  We will do this by investigating obstructions from higher order $\rho$-invariants.  The next definition first appeared in~\cite{Ha2}.

\begin{definition}[Rational derived series]
  For a group $G$, the \emph{rational derived series} is defined inductively as follows.  Let $G^{(0)}_r := G$, and then let $$G^{(k+1)}_r := \ker\left(G^{(k)}_r \to G^{(k)}_r / [G^{(k)}_r,G^{(k)}_r] \to G^{(k)}_r / [G^{(k)}_r,G^{(k)}_r] \otimes_{\Z} \Q \right).$$
\end{definition}


\begin{definition}[Order $n$ signatures]
  For a $3$-manifold $M$ together with a representation $\phi \colon \pi_1(M) \to \Gamma$, the Cheeger-Gromov Von Neumann $\rho$-invariant $\rho^{(2)}(M,\phi) \in \R$ is defined~\cite{ChGr1},~\cite{ChWe}.  Given a $4$-manifold $W$ with $\partial W=M$, such that $\phi$ extends over $\pi_1(W)$, the $\rho$-invariant can be computed as the $L^{(2)}$-signature defect $\sigma^{(2)}(W, \Gamma) - \sigma(W)$, where $\sigma^{(2)}$ is the $L^{(2)}$-signature and $\sigma$ is the ordinary signature of $W$.  The $L^{(2)}$-signature is the signature of the intersection form on $H_2(W;\mathcal{N}\Gamma)$, where $\mathcal{N}\Gamma$ is the Von Neumann algebra of $\Z\Gamma$. For more details see \cite[Section~5]{COT}, \cite{ChaO}.

  For $n \geq 0$, the set of \emph{order $n$ signatures} of a knot $K$ is the set of real numbers given by $\rho$-invariants $\rho^{(2)}(M_K,\phi \colon \pi_1(M_K) \to \Gamma)$, where $M_K$ is the zero surgery manifold, $\phi \colon \pi_1(M_K) \to \pi_1(W) \to \pi_1(W)/\pi_1(W)^{(n+1)}_r =: \Gamma$ and $W$ is an $n$-solution for $K$.
\end{definition}

Note that there is a unique order $0$ signature, where $\Gamma =\Z$, so that the $\rho^{(2)}$-invariant is equal to the average of the Levine-Tristram signature function over $S^1$~\cite[Section~5]{COT}.

\begin{definition}[Extendable branches]
Suppose a knot $K \in \mathcal{G}_n$ but $K \notin \mathcal{G}_{n+1}$, and $K$ bounds a symmetric grope $\Sigma$ of height $n$.  A set of branches is \emph{extendable} if there exists an embedded branch-symmetric grope $\Sigma'$ such that $\Sigma \subset \Sigma'$, and the branches in question are contained in branches of $\Sigma'$ of length $n$ (recall that if all branches were of length $n$ then the whole grope would be a symmetric grope of height $n+1$).
\end{definition}

In previous sections, we used the fact that if a knot $K$ is not in $\mathcal{G}_{n+3}$, then the grope norm satisfies $\|K\|^q \geq 1/(2q)^{n+1}$ by Proposition~\ref{prop:boundedbelow}.  In this section, we will show that order $n$ $\rho$-invariants, as well as obstructing a knot from lying in $\mathcal{G}_{n+3}$, can give lower bounds on the number of \emph{non-extendable} branches of any height~$n+2$ grope.  This can improve our lower bound for the infimum of the grope length function, taken over all gropes of height $n+2$.  Controlling the set of order $n$ $\rho$-invariants uses the existence of a grope of height~$n+2$. It could be that our knot bounds a very simple grope of height $n+1$, and thus we cannot rule out $\|K\|^q = 1/(2q)^{n}$; Corollary~\ref{cor:L2-lower-bound} below gives the precise statement.  Still, Theorem~\ref{theorem:rho-invariant-lower-bound} does represent a refinement of our lower bounds.

In fact, for the proof of Theorem~\ref{theorem:rho-invariant-lower-bound}, we could have begun with a grope of any height.  However the less we restrict, the larger the set of $\rho$-invariants whose infimum gives a lower bound on the number of non-extendable branches. (Taking the infimum over a larger set can of course result in a smaller infimum.)  There are practical advantages to the way we have proceeded.  By restricting to the set of all possible order~$n$ $\rho$-invariants, there exist, at least for examples constructed from iterated infections, techniques which can show the infimum to be nonzero, and indeed arbitrarily large.  We refer to~\cite[Theorem~4.5]{Hor3} for explicit examples.

Denote the set of order $n$ signatures of $K$ by $\mathcal{S}_n(K)$.  The following theorem was inspired by~\cite{Hor3}.

\begin{theorem}\label{theorem:rho-invariant-lower-bound}
Let $n \geq 1$.  Suppose that $K$ bounds a symmetric grope $\Sigma$ of height $n+2$ and first stage genus $g_1(\Sigma)$.  Let $e$ be the maximal cardinality amongst all sets of extendable branches and let $d = g_1(\Sigma)-e$.  Then
\[\inf \{|\rho_n|\, |\, \rho_n \in \mathcal{S}_n(K) \} \leq 4d.\]
\end{theorem}

\begin{proof}
We follow the idea of the proof of Theorem~\ref{thm:refinedslicegenusbounds}.    Let $g = g_1(\Sigma)$ and let $\a_1,\dots,\a_{2g}$ be the basis curves on the first stage surface of $\Sigma$ that bound the higher grope stages.  Perform symmetric surgery on the $\a_i$.
Here we mean that we perform surgery on $D^4$ along a thickening $S^1 \times D^3$ of a push-off of an $\a_i$ curve -- each such surgery adds a connected $S^2 \times S^2$ or $S^2 \wt{\times} S^2$ summand to $D^4$ -- and then we use the core of the resulting pair of surgery disks to reduce the genus of $\Sigma_{1:1}$ with either symmetric surgery.  We obtain a 4-manifold $V$ with a slice disk $\Delta$ for $K$ arising from the surgered first stage of $\Sigma$.  Define $W:= V -\nu \Delta$.

As in \cite[Theorem~8.11]{COT}, and as explained above in the proof of Theorem~\ref{thm:refinedslicegenusbounds}, we obtain a collection of surfaces $S_i, B_i$, $i=1,\dots, 2g$.  The difference from the proof of Theorem~\ref{thm:refinedslicegenusbounds} is that we perform symmetric surgery everywhere in the current proof, whereas before we did asymmetric surgery for the non-extendable branches.  Now we want to look at signatures associated to $(n)$-solutions i.e.\ height $n+2$ gropes only, in order to have a practically useful condition; as remarked above, in favourable situations we actually can show that the set of order $n$ signatures,  restricted in this way, is bounded below.

The surfaces $B_i$ are constructed from the second stage surfaces of $\Sigma$, capped off by the surgery disks.  The surfaces $S_i$ come from the dual surgery spheres, pushed off the contraction and then tubed into the $B_i$ to remove intersections between them that appear while pushing off the contraction.
The surfaces are framed and embedded, generate $H_2(W;\Z) \cong \Z^{4g} = \Z^{4d} \oplus \Z^{4e}$.  The $\Z^{4d}$ summand corresponds to non-extendable branches and the $\Z^{4e}$ summand corresponds to the extendable branches.  The intersection form of $W$ on $H_2(W;\Z)$ is hyperbolic, since the intersections between the surfaces $S_i, B_i$ are $S_i\cdot S_j =0 = B_i \cdot B_j$ and $S_i \cdot B_j = \delta_{ij}$.  The ordinary signature $\sigma(W)$ of $W$ therefore vanishes.

Define $\Gamma := \pi_1(W)/\pi_1(W)^{(n+1)}_r$, the quotient by the $(n+1)^{th}$ rational derived subgroup.  The group $\Gamma$ is poly-torsion-free-abelian (PTFA -- see \cite[Section~2]{COT}) and is therefore amenable and in Strebel's class $D(\Q)$~\cite{Str}.  This will be useful to bound the rank of the $\Z\Gamma$ homology shortly.  Note that $\partial W=M_K$, the zero surgery on $K$, and $W$ is an $n$-solution~\cite[Theorem~8.11]{COT}, so that $\rho^{(2)}(M_K, \phi \colon \pi_1(M_K) \to \Gamma) = \sigma^{(2)}(W,\Gamma) - \sigma(W) = \sigma^{(2)}(W,\Gamma) \in \mathcal{S}_n(K)$ is an $n^{th}$ order signature of~$K$.

The surfaces $S_i, B_i$, for $i = 2d+1,\dots, 2g$, were constructed from spheres and second stage surfaces belonging to branches of length $n+1$ that can be extended to have length at least $n+2$, therefore they satisfy $\pi_1(S_i) \subset \pi_1(W)^{(n+1)}$ and $\pi_1(B_i) \subset \pi_1(W)^{(n+1)}$.
Thus these surfaces lift to the $\Gamma$-cover.  Now work over the Von Neumann algebra $\mathcal{N}\Gamma$.  For $i = 2d+1,\dots, 2g$ the surfaces $S_i,B_i$ define elements of $H_2(W;\mathcal{N}\Gamma)$, in which they generate a submodule $E$, and the intersection form $\lambda_{\Gamma} \colon H_2(W;\mathcal{N}\Gamma) \times H_2(W;\mathcal{N}\Gamma) \to \mathcal{N}\Gamma$ restricted to $E$ is hyperbolic.

\begin{claim*}
  The submodule $E$ is a direct summand.
\end{claim*}
 We have a short exact sequence
 \[0 \to E \to H_2(W;\NG) \to H_2(W;\NG)/E \to 0.\]
The map
\[\ba{rcl}  s \colon H_2(W;\NG) & \to & E \\ y &\mapsto & \sum_{i=2d+1}^{2g} (y \cdot B_i)[S_i] +  \sum_{i=2d+1}^{2g} (y \cdot S_i)[B_i],   \ea\]
which is a homomorphism by linearity of the intersection pairing, splits this exact sequence.  This proves the claim that $E$ is a direct summand.

\begin{claim*}
  The submodule $E$ is free.
\end{claim*}

Suppose that $$C:=\sum_{i=2d+1}^{2g} n_i[S_i] +  \sum_{i=2d+1}^{2g} m_i[B_i] =0.$$
Taking the intersection $C \cdot B_k$ implies that $n_k=0$.  Similarly the intersection $C \cdot S_k =0$ implies that $m_k=0$.  This proves the claim that $E$ is free.
Let $D := H_2(W;\NG)/E$.  We need to investigate the $L^{(2)}$-dimension of $D$.

\begin{claim*}
  We have that $|\rho^{(2)}(M_K,\phi)| \leq \dim^{(2)}(D)$.
\end{claim*}

Although $D$ may not be a free $\NG$-module, following~\cite[pp.~4776--80]{Cha-2014-2}, we can define the $L^{(2)}$-signature $\sigma^{(2)}(W,\Gamma)$ by replacing $D$ with $P(D) = D/T(D)$, a projective quotient of $D$, where \[T(D)= \{x \in D \,|\, f(x)=0 \text{ for any homomorphism }f \colon D \to \NG\}\]
is the Von Neumann torsion submodule of $D$~\cite[p.~239]{Luc}.  The submodule $T(D)$ satisfies $\dim^{(2)}(T(D)) = 0$ so $\dim^{(2)}(P(D)) = \dim^{(2)}(D)$.  Note that $T(E \oplus D) = T(D)$ since $E$ is free.  Moreover, for any $t \in T(D)$ and for any $x \in E \oplus D$ we have $\lambda_W(t,x)=0$.  To see this note that $T(-)$ is functorial, and so the adjoint of $\lambda_W$ sends $t$ into $T((E\oplus D)^*)$.  But $T(A^*)=0$ for any $\NG$-module $A$ by \cite[Lemma~3.4]{Cha-2014-2}.  It follows that $\lambda_W(t,-) = 0 \in (E \oplus D)^*$.

Find an $\NG$-module $Q$ such that $P(D)\oplus Q \cong \NG^\ell$ is free.  Extend the intersection form $\lambda_W$ from $E \oplus P(D)$ to $E \oplus \NG^{\ell}$ by having it vanish on $Q$; we also denote the extended intersection form on $E \oplus \NG^\ell$ by $\lambda_W$.  We can then define the $L^{(2)}$-signature $\sigma^{(2)}(\lambda_W)$ as usual using the functional calculus~\cite[Section~5]{COT}, ~\cite[pp.~4779--80]{Cha-2014-2}.  We note that $|\sigma^{(2)}(\lambda_W|_{\NG^\ell})| \leq \dim^{(2)}(D)$, since the absolute value of the $L^{(2)}$ signature is always bounded above by the $L^{(2)}$-dimension of the underlying module, and extending by the zero form on $Q$ only increases the dimension of the $0$-eigenspace, leaving the dimensions of the positive and negative definite subspaces unaltered.  Now use that $\lambda_W|_{E}$ is nonsingular, and that it is defined on a free module, to change basis so that the direct sum decomposition $E \oplus \NG^\ell$ is orthogonal with respect to the intersection form.
This, together with the facts that the intersection form $\lambda_W|_E$ is hyperbolic and $\sigma(W)=0$, yields:
\[|\rho^{(2)}(M_K, \phi)| = |\sigma^{(2)}(W, \Gamma)|  = |\sigma^{(2)}(\lambda_W|_{\NG^\ell})| \leq \dim^{(2)} D.\]
Now we have one last claim.

\begin{claim*}
  We have $\dim^{(2)}(D) \leq 4d$.
\end{claim*}

To prove the claim we will use
the following dimension bound for homology over $\mathcal{N}\Gamma$ in terms of the dimension of the $\Q$ homology, which can be found in \cite[Theorem~3.11]{Cha-2014-2}.

\begin{thmn}[3.11 of \cite{Cha-2014-2}]\label{thm:dimension-bound}
Suppose $G$ is amenable and in $D(R)$ with $R = \Q$ or $\Z_p$, and $C_*$ is a projective
chain complex over $\Z G$ with $C_m$ finitely generated.
 If $\{x_i\}_{i \in I}$ is a collection of $m$-cycles in $C_m$, then for the submodules
$H \subset H_m(\mathcal{N}G \otimes_{\Z G} C_*)$ and $\ol{H} \subset  H_m(R \otimes_{\Z G} C_*)$ generated by $\{[1_{\mathcal{N}G} \otimes x_i]\}_{i\in I}$ and
$\{[1_{R} \otimes x_i]\}_{i\in I}$ respectively, we have:
\[\dim^{(2)} H_m(\mathcal{N}G \otimes_{\Z G} C_*) - \dim^{(2)} H \leq  \dim_R H_m(R \otimes_{\Z G} C_*) - \dim_R \ol{H}.\]
\end{thmn}

As noted above, $\Gamma=G$ satisfies the hypothesis of Theorem~\ref{thm:dimension-bound}.  Apply the theorem with $R=\Q$, $G=\Gamma$ and $m=2$. Let $C_* = C_*(W;\Z\Gamma)$ and take the $x_i$ to be given by the lifts to the $\Gamma$-cover of the surfaces $S_i, B_i$ corresponding to extendable branches. We have $\dim_{\Q} H_2(\Q \otimes_{\Z \Gamma} C_*(W;\Z\Gamma)) = \dim_{\Q} H_2(W;\Q) = 4g$, $\dim_{\Q} \ol{H} = 4e$ and $\dim^{(2)} H = \dim^{(2)} E = 4e$.  Thus Theorem~\ref{thm:dimension-bound} and additivity of the $L^{(2)}$ dimension imply that
\begin{align*}
 \dim^{(2)} D = \dim^{(2)} H_2(W;\mathcal{N}\Gamma) - \dim^{(2)} E &  \leq \dim_{\Q} H_2(W;\Q) - \dim_{\Q} \Q \otimes E\\  &= 4g -4e = 4d.
\end{align*}
This completes the proof of the claim that $\dim^{(2)}(D) \leq 4d$. Combined with the inequality $|\rho^{(2)}(M_K, \phi)| \leq \dim^{(2)}(D)$ proved above, this completes the proof of the Theorem~\ref{theorem:rho-invariant-lower-bound}, by the following logic.  We have shown that given a height $n+2$ symmetric grope with $d$ non-extendable branches, there exists an order~$n$ $\rho^{(2)}$-invariant which is at most $4d$.  Considering all possible gropes of height $n+2$, they correspond to possibly different representations $\phi \colon \pi_1(M_K) \to \Gamma$, so the infimum of the order~$n$ signatures becomes a lower bound for the number of non-extendable branches of any height $(n+2)$ grope.

\end{proof}

\begin{corollary}\label{cor:L2-lower-bound}
For each $n\geq 0$ we have
  \begin{align*} \frac{\inf_{\rho_n \in \mathcal{S}_n(K)} \{|\rho_n|\}}{4(2q)^{n+1}} &\leq
  \inf \left\{ \|\Sigma\|^q \, \begin{array}{|ll} K \text{\emph{ bounds a branch-symmetric grope }}\Sigma \text{
 \emph{all of whose}} \\  \text{\emph{branches have length at least }} n+1 \end{array}
  \right\} \\
  &\leq \inf \{\|\Sigma\|^q \, | \, K
 \text{\emph{ bounds a symmetric grope }}\Sigma \text{ \emph{of height} } \geq n+2\}.\end{align*}
Thus, for each $n\geq 0$, we have
\[\min \left\{ \frac{\inf_{\rho_n \in \mathcal{S}_n(K)} \{|\rho_n|\}}{4(2q)^{n+1}}, \frac{1}{(2q)^{n}} \right\} \leq  \|K\|^q.\]
\end{corollary}

\begin{proof}
Let $\Sigma$ be a branch-symmetric grope where each branch has length at least $n+1$ and let $m$ be the number of branches of length $n+1$.  Then $\Sigma$ contains a subgrope $\Sigma^\prime$ that is a symmetric grope of height $n+2$.   If $d$ is the number of non-extendable branches of $\Sigma^\prime$, then $d\leq m$.  By the previous theorem, $\Sigma^\prime$ must have at least $\frac{\inf_{\rho_n \in \mathcal{S}_n(K)} \{|\rho_n|\}}{4}$ non-extendable branches.  Each branch of $\Sigma$ of length $n+1$ contributes at least $1/(2q)^{n+1}$ to $\|\Sigma^\prime\|^q$.   Thus we see that
\[
\frac{\inf_{\rho_n \in \mathcal{S}_n(K)} \{|\rho_n|\}}{4(2q)^{n+1}} \leq \frac{d}{(2q)^{n+1}} \leq \frac{m}{(2q)^{n+1}} \leq \|\Sigma\|^q.
\]
This completes the proof of the first statement of the corollary.

For the second part, suppose $K$ bounds a branch-symmetric grope and let $n \geq 0$ be fixed.  If all the branches of $\Sigma$ have length at least $n+1$ then by the proof of the first part, we see that $\frac{\inf_{\rho_n \in \mathcal{S}_n(K)} \{|\rho_n|\}}{4(2q)^{n+1}} \leq \|\Sigma\|^q$.  Otherwise, there is some branch of length at most $n$, so by Proposition~\ref{prop:boundedbelow} we have $\|\Sigma\|^q \geq 1/(2q)^{n}$.
\end{proof}

\section{Quasi-isometries to knot concordance with the slice genus metric}\label{section:quasi-isometries}

Recall that a map of metric spaces $f \colon (X,d_X) \to (Y,d_Y)$ is a \emph{quasi-isometry} if there exist constants $A \geq 1$ and $B \geq 0$ such that
\[ \frac{1}{A}d_X(x,y) -B  \leq d_Y(f(x),f(y)) \leq A\cdot d_X(x,y) +B \]
for any $x,y \in X$, and if there is a constant $C \geq 0 $ such that for any $z \in Y$ there exists an $x\in X$ such that $d_Y(z,f(x)) \leq C$.

Let $d_s$ be the metric on $\C$ defined by the slice genus, which was studied in detail in~\cite{Cochran-Harvey:2014-1}.  Let $\|\text{--}\|_s$ be the associated norm.
Not only are the slice and grope metric spaces not isometric, but we show that the identity is not even a quasi-isometry.   It is unknown whether there is another quasi-isometry between them, but we do not expect one to exist.

\begin{proposition}
  For any $q \geq 1$, neither of the identity maps $\Id \colon (\C, d_s) \to (\C, d^q)$ nor $\Id \colon (\C, d^q) \to (\C, d_s)$ are quasi-isometries.
\end{proposition}

\begin{proof}
  First we will show that for any $A \geq 1$ and for any $B\geq 0$ there exists a knot $K$ such that
\[ \frac{1}{A}\|K\|_s -B  > \|K\|^q.\]
  From this, we see that the left hand side of the condition for $\Id \colon (\C, d_s) \to (\C, d^q)$ to be a quasi-isometry is violated.  Since $q<q'$ implies that $\|K\|^q \leq \|K\|^{q'}$, it suffices to show this for $q=1$.
  Let $J$ be a knot with slice genus one such that $\tau(J)=1$ (for example a trefoil), where $\tau \colon \C \to \Z$ is the invariant from knot Floer homology~\cite{Ozsvath-Szabo-2003}.  Let $A,B$ be given as above.  Choose $m\in \mathbb{N}$ such that $m > A(B+1)$.  Then choose $n$ such that $m/2^n <1$; that is choose $n\in \mathbb{N}$ which is greater than $\log m /\log 2$.  Now define $K$ to be a connect sum of $m$ copies of the $n$-fold iterated positive Whitehead double of $J$, $K:= \#^m \Wh^n_+(J)$.  According to \cite{Hedden-tau}, $\tau(\Wh^n_+(J))=1$, so by additivity $\tau(K) =m$.  Therefore $\|K\|_s \geq m$ since $\tau(K) \leq \|K\|_s$ by \cite{Ozsvath-Szabo-2003}.  Thus the left hand side of the inequality above satisfies
  \[ \frac{1}{A}\|K\|_s -B \geq \frac{1}{A}m -B > \frac{1}{A}(A(B+1)) -B = 1.\]
  On the other hand, $\Wh^n_+(J)$ bounds a grope of height $n+1$, wherein all the stages are of genus one, and $K$, the connect sum of $m$ copies of $\Wh^n_+(J)$, bounds a grope of height $n+1$, with first stage genus $m$ and all higher stages genus one.  Therefore $\|K\|^1 \leq m/2^n < 1$, since the $q=1$ length of a grope of height one all of whose surfaces are genus one is $1/2^n$.  This shows that $K$ has the property desired.

  To show that the inverse identity map $\Id \colon (\C, d^q) \to (\C, d_s)$ is not a quasi-isometry, we prove that the right hand side in the defining condition is not satisfied.  That is, for any $A,B$ there exists a knot $K$ with
$\|K\|_s > A\|K\|^q +B$.  This is equivalent to $\|K\|_s/A - B/A > \|K\|^q$, and thus we can apply the argument above with $B/A$ replacing~$B$.

\end{proof}

\bibliographystyle{alpha}
\def\MR#1{}
\bibliography{Timbib070815}

\end{document}